\documentclass[11pt,reqno]{amsart}
\usepackage{amsmath}
\usepackage{amssymb}
\usepackage{mathrsfs}
\usepackage{mhequ}
\usepackage{color}
\usepackage{graphicx}
\usepackage{enumerate}
\usepackage{todonotes}
\usepackage[unicode=true,
bookmarks=true,bookmarksnumbered=false,bookmarksopen=false,
breaklinks=false,pdfborder={0 0 1},backref=false,colorlinks=true]{hyperref}
\usepackage{bbm}
\usepackage[small]{caption}
\usepackage{hyphenat}
\usepackage{enumitem}
\usepackage{tikz}
\usetikzlibrary{arrows.meta, calc, fit}
\usepackage{float}

\newtheorem{maintheorem}{Theorem}

\newtheorem{theorem}{Theorem}[section] 
\newtheorem{lemma}[theorem]{Lemma}   
\newtheorem{corollary}[theorem]{Corollary}
\newtheorem{proposition}[theorem]{Proposition}
\newtheorem{definition}[theorem]{Definition}
\newtheorem{remark}[theorem]{Remark}
\newtheorem{example}[theorem]{Example}

\newcommand{\N} {\mathbb N}

\begin{document}
	\title[On Foliated Optimal Transport]%
	{On Foliated Optimal Transport} 
	
	\author[D. S. Ledesma]{Diego S. Ledesma}
	\address{Diego Sebastian Ledesma\\
		Institute of Mathematics, Department of Mathematics\\
		Universidade Estadual de Campinas\\
		13.083-859 Campinas - SP\\
		Brazil.
	}
	\email{ledesma@ime.unicamp.br}
	
	\author[R. Possobon]{Renata Possobon}
	\address{Renata Possobon\\
		Institute of Mathematics, Department of Mathematics\\
		Universidade Estadual de Campinas\\
		13.083-859 Campinas - SP\\
		Brazil.
	}
	\email{possobon@ime.unicamp.br}

	\author[C.S. Rodrigues]{Christian S. Rodrigues}
	\address{Christian S. Rodrigues\\
		Institute of Mathematics, Department of Applied Mathematics\\
		Universidade Estadual de Campinas \\
		13.083-859 Campinas - SP\\
		Brazil \\
		and Max-Planck-Institute for Mathematics in the Sciences\\
		Inselstr. 22\\
		04103 Leipzig\\
		Germany.
	}
	\email{rodrigues@ime.unicamp.br}
	
	
	\begin{abstract} 
		 In this paper we introduce a geometrically constrained optimal transport problem on Riemannian manifolds endowed with a Borel foliation, in which mass is transported only within individual leaves. We characterize the existence of foliated transport plans through the equality of the transverse projections of the marginals and show that the corresponding quadratic transport cost decomposes into a measurable family of leafwise optimal transport problems. We also formulate the problem as a classical Kantorovich problem with an extended-valued leafwise cost. Finally, we introduce a foliated continuity equation and establish a Benamou-Brenier formula relating the static and dynamical foliated formulations.
	\end{abstract}
	
	
	
	\maketitle
	
	\section{Introduction}

	Optimal transport provides a natural framework for investigating how the geometry of a space $X$ is reflected in the space $\mathscr P(X)$ of probability measures on it. Distances, admissible trajectories and geometric constraints on the underlying space induce corresponding structures on spaces of probability measures, although the precise way in which this information is transferred is often nontrivial \cite{Vil03, Vil09}. This relationship becomes particularly interesting when $X$ carries a distinguished decomposition into geometric layers. Such a decomposition naturally leads to the study of transport processes in individual layers. If we impose the geometric constraint that mass may move along these layers but cannot be transferred from one layer to another, then the resulting global problem is governed both by the geometry internal to each layer and by the way in which the total mass is distributed among them. This perspective provides a setting in which global transportation can be understood through a measurable family of lower-dimensional transport problems. In this work, this point of view is considered when the layers of $X$ are the leaves of a foliation $\mathcal F$. We impose the constraint that no mass can be transported between distinct leaves. Thus, the foliation does not merely alter the cost of transportation, but also determines which pairs of points may interact through an admissible transport plan. The resulting problem will be referred to as foliated optimal transport.
	
	A useful heuristic picture is provided by the distinction between laminar and turbulent motion. In a laminar flow, particles move along coherent layers, with no substantial mixing between neighboring layers. Turbulent motion, by contrast, is characterized by mixing and by trajectories that interact across different directions and scales. Foliated transport resembles the laminar picture: the leaves play the role of prescribed flow layers, and admissible transport is required to preserve them. Classical unconstrained optimal transport allows a coupling to connect arbitrary points of the ambient space, whereas foliated transport suppresses this transverse mixing by allowing interactions only between points lying on the same leaf. This analogy is intended only as geometric intuition, but it captures the central distinction between the two problems.
	
	To formalize the idea, we consider a Riemannian manifold $X$ endowed with a Borel foliation $\mathcal F$, in the sense that there exist a standard Borel space $X^*$ and a Borel map $\pi:X \to X^*$ whose fibers are precisely the leaves of $\mathcal F$. The space $X^*$ parametrizes the leaves and describes the transverse direction of the foliation. A transport plan between two probability measures $\mu$ and $\nu$ on $X$ is called foliated when it is concentrated on pairs $(x,x')$ satisfying $\pi(x)=\pi(x')$. The interchange from a decomposition of $X$ to a corresponding decomposition of probability measures is made possible by a disintegration theorem. Since this result will be used repeatedly, we recall a suitable version below.
	
	\
	
	\begin{theorem} [{\cite{Rok52, PR25}}] \label{theorem:disintegration}
		Let $X$ and $Y$ be standard Borel spaces, let $\pi: X \to Y$ be a Borel map. Given $\mu \in \mathscr P(X)$, set $\sigma:=\pi_*\mu$, that is, $\sigma$ is the push-forward of $\mu$ with respect to $\pi$. Then there exists a family of probability measures	$\{\mu^y\}_{y\in Y}$ on $X$ such that,
		\begin{enumerate}[leftmargin=*]
			\item $\mu^{y}(\pi^{-1}(\{y\}))=1$ for $\sigma$-almost every $y \in Y$;
			\item the map $y \mapsto \mu^y(E)$ is Borel, for every Borel subset $E \subset X$;
			\item $\mu(E)=\int \mu^{y}(E) ~d\sigma (y)$, for every Borel subset $E \subset X$.
		\end{enumerate}
		Moreover, the family $\{\mu^y\}_{y\in Y}$ is uniquely determined up to a $\sigma$-negligible set.
	\end{theorem}
	
	\
	
	\noindent In our context, the measure $\sigma$ describes how the total mass is distributed over the parameter space $Y = X^*$, while each conditional measure $\mu^y$ describes the distribution of mass inside the corresponding fiber $\mathcal L_y:=\pi^{-1}(\{y\})$. We show in Proposition \ref{Propositon_foliated_plans} that the existence of a foliated transport plan between $\mu$ and $\nu$ is equivalent to $\pi_*\mu=\pi_*\nu$. By Theorem \ref{theorem:disintegration}, we have
	\begin{displaymath}
		\mu=\int_{X^*}\mu^y\,d\sigma(y),
		\qquad
		\nu=\int_{X^*}\nu^y\,d\sigma(y),
	\end{displaymath}
	where, for $\sigma$-almost every $y$, the measures $\mu^y$ and $\nu^y$ are concentrated on the leaf $\mathcal L_y$. Therefore, the global transport problem can be analyzed through transport problems between $\mu^y$ and $\nu^y$.
	
	Our first main result makes this decomposition precise. To state it, let $\Pi_{\mathcal F}(\mu,\nu)$ denote the set of foliated transport plans between $\mu$ and $\nu$. For points $x,x'$ in the same leaf, let $d_{\mathcal L}(x,x')$ be their intrinsic Riemannian distance within the leaf, and set $d_{\mathcal L}(x,x')=+\infty$ for points in distinct leaves. The foliated quadratic transport cost is
	\begin{displaymath}
		W_{2,\mathcal F}^2(\mu,\nu):=\inf_{\gamma\in\Pi_{\mathcal F}(\mu,\nu)}\int_{X\times X}d_{\mathcal L}^2(x,x')\,d\gamma(x,x'),
	\end{displaymath}
	and $W_{2,\mathcal L_y}^2$ denotes the quadratic transport cost on the leaf $\mathcal L_y$ computed with $d_{\mathcal L}$. The precise definitions are given in Section \ref{Section:Kantorovich}.
	
	\
	
	\begin{maintheorem} \label{maintheorem:A}
		Let $X$ be a Riemannian manifold with a Borel foliation $\mathcal F$ of class $C^1$. Let $\mu,\nu\in\mathscr P(X)$. Then there exists a foliated transport plan between $\mu$ and $\nu$ if and only if $\pi_*\mu=\pi_*\nu$. In this case, setting $\sigma:=\pi_*\mu$ and letting $\{\mu^y\}_{y\in X^*}$ and $\{\nu^y\}_{y\in X^*}$ be disintegrations of $\mu$ and $\nu$ with respect to $\pi$, we have
		\begin{displaymath}
			W_{2,\mathcal F}^{2}(\mu,\nu)=\int_{X^*}W_{2,\mathcal L_y}^{2}(\mu^y,\nu^y)\,d\sigma(y),
		\end{displaymath}
		where both sides may be equal to $+\infty$. If, moreover, $X$ is complete and $W_{2,\mathcal F}^2(\mu,\nu)<+\infty$, then there exists an optimal foliated transport plan, that is, a plan $\gamma^\star\in\Pi_{\mathcal F}(\mu,\nu)$ with $\int_{X\times X}d_{\mathcal L}^2\,d\gamma^\star=W_{2,\mathcal F}^2(\mu,\nu)$.
	\end{maintheorem}
	
	\
	
	Theorem \ref{maintheorem:A} combines Proposition \ref{Propositon_foliated_plans}, Theorem \ref{Theorem-2}, Lemma \ref{lemma:dL-Borel-lsc} and Corollary \ref{corollary:existence-optimal-plans}. It shows that foliated optimal transport is a measurable family of independent optimal transport problems on the leaves. The common transverse measure $\sigma$ specifies how these problems are assembled into a global transport problem, while the conditional measures determine the transportation performed inside each individual leaf.
	
	The leafwise constraint also admits an equivalent formulation as a classical Kantorovich problem with the extended-valued cost $d_{\mathcal L}^2$, Theorem \ref{theorem:extended-cost-formulation}. Since $d_{\mathcal L}=+\infty$ for points in distinct leaves, plans that move a positive amount of mass between distinct leaves automatically have infinite cost. Thus, the geometric restriction may be expressed either through the admissible class of transport plans or directly through the cost function. This formulation is what yields the existence of optimal foliated transport plans in Theorem \ref{maintheorem:A}, through the classical existence theory for lower semicontinuous costs. Indeed, we show in Lemma \ref{lemma:dL-Borel-lsc} that $d_{\mathcal L}$ is always Borel, and that it is lower semicontinuous whenever $X$ is complete.

	We also develop a dynamical formulation of the problem. In the classical Benamou-Brenier theory, transport is described by a curve of probability measures $(\mu_t)_{t\in[0,1]}$ and a velocity field $v_t$ satisfying a continuity equation, and the classical cost for the transport is recovered by minimizing the kinetic action. In the foliated setting, we impose $v_t(x)\in T_x\mathcal L_x$. This tangency condition is the dynamical counterpart of requiring a transport plan to be concentrated on pairs of points belonging to the same leaf. It also implies the conservation law $\pi_*\mu_t=\pi_*\mu_0$ throughout the evolution: mass may be rearranged inside the leaves, but its transverse distribution cannot change. Minimizing the kinetic action over such solutions defines the dynamical foliated cost
	\begin{displaymath}
		\mathcal W_{2,\mathcal F}^2(\mu_0,\mu_1):=\inf\left\{\int_0^1\int_X\|v_t(x)\|_g^2\,d\mu_t(x)\,dt:(\mu_t,v_t)\in\operatorname{CE}_{\mathcal F}(\mu_0,\mu_1)\right\},
	\end{displaymath}
	where $g$ is the Riemannian metric of $X$ and $\operatorname{CE}_{\mathcal F}(\mu_0,\mu_1)$ denotes the set of solutions of the continuity equation joining $\mu_0$ to $\mu_1$ whose velocity fields are tangent to the leaves. The precise definitions are given in Section \ref{Section:Benamou}.
	
	\
	
	\begin{maintheorem} \label{maintheorem:B}
		Let $X$ be a complete Riemannian manifold endowed with a Borel foliation $\mathcal F$ of class $C^1$. Then, for every $\mu_0,\mu_1\in\mathscr P_2(X)$,
		\begin{displaymath}
			\mathcal W_{2,\mathcal F}^{2}(\mu_0,\mu_1)=W_{2,\mathcal F}^{2}(\mu_0,\mu_1).
		\end{displaymath}
	\end{maintheorem}
	
	\
	
	Theorem \ref{maintheorem:B} is the foliated Benamou-Brenier formula, proved in Section \ref{Section:Benamou} as Theorem \ref{theorem:foliated-benamou-brenier}. Its proof relies on two representation results adapted to the foliated setting, a leafwise superposition principle, Lemma \ref{lemma:leafwise-superposition}, and an Eulerian representation of measures on leafwise curves, Lemma \ref{lemma:eulerian-representation}. It also uses a Borel selection of leafwise minimizing geodesics, Lemma \ref{lemma:borel-geodesic-selection}, which exists on every complete manifold by the compactness argument behind Lemma \ref{lemma:dL-Borel-lsc}.
	
	The static and dynamical formulations reveal complementary aspects of the same geometric structure. In the static description, the foliation restricts the pairs of points that may be coupled. In the dynamical description, it restricts the admissible directions of motion. The superposition principle connects these viewpoints by representing solutions of the continuity equation through measures on trajectories, while measurable selections of leafwise transport plans and minimizing geodesics allow the individual leafwise problems to be assembled globally.
	
	\
	
	\noindent \emph{Related work.} Optimal Transport with couplings restricted to fibers has recently been studied in product and bundle settings. An earlier instance appears in the work of Gigli, who considered a distance between probability measures on the tangent bundle of a Riemannian manifold with a prescribed projection onto the manifold, defined through couplings that only pair vectors based at the same point \cite{Gig11}. Peszek and Poyato \cite{PP23} studied a fibered optimal transport distance on probability measures over $\mathbb R^d\times\mathbb R^d$, whose cost is infinite between distinct fibers, showed that the resulting space is Polish, and endowed it with a weak Riemannian structure in the spirit of Otto's formal calculus. They also characterized absolutely continuous curves through a continuity
	equation along the fibers, with estimates relating the velocity field to
	the metric derivative \cite{PP23}. Chemseddine, Hagemann, Steidl and Wald \cite{CHSW25} defined a conditional Wasserstein distance on product spaces through restricted couplings, proved that it equals the expected Wasserstein distance between the conditional measures, and showed that the displacement interpolation of an optimal restricted plan is a geodesic whose velocity field transports no mass across fibers and solves the continuity equation. Kitagawa and Takatsu \cite{KT24} introduced a family of disintegrated Monge-Kantorovich metrics on metric fiber bundles, which includes fibered Wasserstein distances, and proved that these metrics are complete and geodesic, with a dual representation. They also showed that, when both exponents coincide, these metrics are given by a transport problem whose cost is infinite between distinct fibers, a bundle analogue of Theorem \ref{theorem:extended-cost-formulation}. In their setting all fibers are isometric to a fixed model space, and the cost on each fiber is the restriction of the ambient distance. For product foliations, the leafwise decomposition in Theorem \ref{maintheorem:A} and one of the two inequalities in Theorem \ref{maintheorem:B} thus have counterparts in \cite{CHSW25}. The present setting differs from these works in four respects. First, the leaves of a Borel foliation need not be the fibers of a product or of a locally trivial bundle, as Example \ref{example:S3-foliation} shows. Second, the leaf space is only assumed to be a standard Borel space. Third, the cost is given by the intrinsic Riemannian distance of the leaves. Finally, to our knowledge, the full equivalence between the static and dynamical formulations, obtained here through a leafwise superposition principle, has no counterpart in these works.
	
	On the dynamical side, optimal transport with velocities constrained to a distribution has been studied on sub-Riemannian manifolds \cite{AL09, FR10}. There the distribution is bracket-generating, so any two points can be joined by admissible curves, and the constraint affects the cost but not the admissible couplings. In the foliated setting the distribution is integrable, admissible curves remain in a single leaf, and the constraint restricts the support of transport plans.
	
	\
	
	The paper is organized as follows. In Section \ref{Section:Kantorovich}, we recall the classical formulation of the optimal transport problem and introduce Borel foliations and foliated transport plans. Proposition \ref{Propositon_foliated_plans} characterizes the existence of foliated transport plans, and Proposition \ref{proposition_2} shows that every foliated transport plan admits a disintegration into transport plans acting on the individual leaves. Lemma \ref{lemma:dL-Borel-lsc} shows that the extended leafwise distance is Borel, and lower semicontinuous when $X$ is complete. We prove in Theorem \ref{Theorem-2} that the foliated quadratic transport cost decomposes as the integral of the corresponding optimal transport costs on the leaves. Theorem \ref{theorem:extended-cost-formulation} provides an equivalent formulation of this problem as a classical Kantorovich problem with an extended-valued leafwise cost, from which Corollary \ref{corollary:existence-optimal-plans} deduces the existence of optimal foliated transport plans. Together, these results prove Theorem \ref{maintheorem:A}. We conclude the section with examples arising from smooth submersions and locally free group actions. In Section \ref{Section:Benamou}, we introduce the foliated continuity equation and the associated dynamical transport cost. Lemma \ref{lemma:leafwise-superposition} is a leafwise superposition principle, representing solutions of the foliated continuity equation by measures on leafwise curves, and Proposition \ref{proposition:transverse-conservation} deduces that the transverse marginal is preserved along every admissible evolution. Lemma \ref{lemma:eulerian-representation} provides the converse representation, from measures on leafwise curves to solutions of the foliated continuity equation. Lemma \ref{lemma:borel-geodesic-selection} provides a Borel selection of leafwise minimizing geodesics. Finally, Theorem \ref{theorem:foliated-benamou-brenier} establishes the foliated Benamou-Brenier formula, which is Theorem \ref{maintheorem:B}, proving the equivalence between the static and dynamical formulations of foliated quadratic transport.

	\

	\section{Monge-Kantorovich problem} \label{Section:Kantorovich}

		
		The foundations of Optimal Transport Theory can be traced back to a logistical problem formulated by Gaspard Monge in 1781. The central idea is to transport a given mass distribution from one location to another while minimizing the cost. More precisely, let $X$ and $Y$ be complete separable metric spaces, and let $\mu \in \mathscr{P}(X)$ and $\nu \in \mathscr{P}(Y)$ be probability measures on $X$ and $Y$, respectively. The problem consists in finding a measurable map $T \colon X \to Y$ such that $T_*\mu = \nu$ and that minimizes the total transportation cost, $\int_X c(x,T(x)) ~d\mu(x)$, where $c \colon X \times Y \to [0,+\infty]$ is a Borel cost function. This formulation, known as the Monge Problem, describes transportation via a deterministic map. In general, an optimal transport map need not exist for arbitrary probability measures and cost functions.
		
		The Kantorovich formulation provides a relaxation of the Monge Problem. More precisely, we define a transport plan between $\mu$ and $\nu$ as a probability measure $\gamma \in \mathscr{P} (X \times Y)$ whose first and second marginals are $\mu$ and $\nu$, respectively. We denote the set of all such plans by
		\begin{displaymath}
			\Pi(\mu,\nu) = \left\{ \gamma \in \mathscr{P}(X \times Y): (\operatorname{proj}_1)_* \gamma = \mu, ~ (\operatorname{proj}_2)_*\gamma = \nu	\right\},
		\end{displaymath}
		where $\operatorname{proj}_1(x,y) = x$, $\operatorname{proj}_2(x,y) = y$. The Monge-Kantorovich Problem is
		\begin{displaymath}
			\inf_{\gamma \in \Pi(\mu,\nu)} \int_{X\times Y} c(x,y) ~d\gamma(x,y).
		\end{displaymath}
		Here, for measurable sets $A\subset X$ and $B\subset Y$, $\gamma (A \times B)$ represents the amount of mass transported from $A$ to $B$. Under certain assumptions on the space and on the cost function, the Monge-Kantorovich problem admits an optimal transport plan.
		
		A particularly important example is given when we define both measures on the same complete separable metric space $X$, and we choose the quadratic cost $c(x,y)=d^2(x,y)$. In this case, the corresponding optimal transport cost defines a metric on the space of probability measures with finite second moment, known as the $2$-Wasserstein distance. More precisely, we define:	
		
		\
	
		\begin{definition}
		Let $(X, d)$ be a separable complete metric space. Consider $p \in [1, \infty)$ and probability measures $\mu$ and $\nu$ on $X$. The Wasserstein distance of order $p$ between $\mu$ and $\nu$ is given by
		\begin{displaymath}
			W_{p}(\mu, \nu):= \Big( \inf\limits_{\gamma \in \Pi(\mu, \nu)}  \int d^p(x_{1}, x_{2}) ~d\gamma (x_{1}, x_{2})\Big)^{\frac{1}{p}}.
		\end{displaymath}
	\end{definition}
	
	\
	
	The same formula makes sense on any separable metric space, complete or not, and it is in this sense that it is applied below to the leaves endowed with their intrinsic distance. In general, $W_{p}$ is not a distance in the strict sense, because it can take the value $+\infty$. To avoid this issue, it is natural to restrict $W_p$ to a suitable subset of probability measures on which it is finite.
	
	\
	
	\begin{definition}
		The Wasserstein space of order $p$ is defined by
		\begin{displaymath}
			\mathscr{P}_{p}(X):= \Big\{ \mu \in \mathscr{P}(X) : \int d^p(x, \tilde{x}) \mu(dx) < +\infty \Big\}
		\end{displaymath}
		for $\tilde{x} \in X$ arbitrary, endowed with $W_p$. 
	\end{definition}
	
	\

	The Monge-Kantorovich Problem does not specify how the mass is transported. Apart from the marginal constraints, the support of a transport plan is not subject to any prescribed geometric restriction. In many situations, however, the geometry of the ambient	space may impose constraints on which transport trajectories are physically	or geometrically admissible. 
	
	A natural and geometrically meaningful example arises when the ambient
	space is equipped with a foliation. In this case, the leaves provide a
	distinguished family of subspaces along which transport can be performed.
	The requirement that mass be transported only inside individual leaves can
	therefore be interpreted as a geometric restriction on the admissible
	transport plans.
	
	Let $(X,d)$ be a Riemannian manifold, with Riemannian metric $g$ and associated distance $d$, endowed with a foliation $\mathcal{F}$ of class $C^1$, of dimension $k$ and codimension $\ell=\dim (X) - k$. That is, $X$ admits an atlas of $C^1$ foliated charts $\phi:U\to D^k\times D^\ell$, where $D^k\subset\mathbb R^k$ and $D^\ell\subset\mathbb R^\ell$ are open balls, whose plaques $\phi^{-1}(D^k\times\{z\})$, $z\in D^\ell$, are connected open subsets of leaves \cite{CC00}. Each leaf, endowed with its own manifold structure, is an injectively immersed $C^1$ submanifold of $X$, and we equip it with the Riemannian metric induced by $g$. A curve in $X$ is called absolutely continuous if it is absolutely continuous in local charts. Throughout this paper, all manifolds are assumed to be finite-dimensional, Hausdorff, second-countable, and without boundary. The ambient Riemannian manifold $X$ is assumed to be connected. Every manifold considered below is Polish with respect to its manifold topology, and its Borel	sigma-algebra therefore defines a standard Borel space. This does not require the chosen Riemannian distance to be complete. Whenever completeness of $X$ is assumed, it refers to the distance $d$ induced by the fixed Riemannian metric $g$. Suppose that $\mathcal F$ is Borel, in the following sense.
	
	\
	
	\begin{definition}[Borel foliation] \label{Def-Borel-foliation}
		A foliation $\mathcal{F}$ of a Riemannian manifold $(X,d)$ is Borel if there exists a standard Borel space $X^*$ and a Borel map $\pi: X \to X^*$ such that $\pi (x) = \pi (x')$ if and only if $x, x'$ belong to the same leaf of $\mathcal F$.
	\end{definition}

\
	
	\begin{remark} \label{remark:scope-Borel-foliation}
		Definition \ref{Def-Borel-foliation} amounts to requiring that ``$x$ and $x'$ lie in the same leaf'' is a smooth equivalence relation, in the sense of descriptive set theory \cite{HKL90}. This is a genuine restriction. For instance, let $\mathcal F$ be the linear foliation of $\mathbb T^2=\mathbb R^2/\mathbb Z^2$ by lines of direction $(1,\theta)$, with $\theta$ irrational. Then $\mathcal F$ is not Borel in the sense od Definition \ref{Def-Borel-foliation}.~ Indeed, the circle $C:=\{0\}\times(\mathbb R/\mathbb Z)$ meets every leaf, and each leaf meets $C$ in an orbit of the rotation $R_\theta(s)=s+\theta$. Suppose that a Borel map $\pi:\mathbb T^2\to {(\mathbb{T}^2)}^*$ as in Definition \ref{Def-Borel-foliation} exists, and let $m$ be the Lebesgue measure on $C$. For every Borel set $B\subset {(\mathbb{T}^2)}^*$, the set $C\cap\pi^{-1}(B)$ is $R_\theta$-invariant, so its $m$-measure is $0$ or $1$ by the ergodicity of $R_\theta$ \cite{Wal82}. Let $\{B_n\}_{n\in\N}$ be a countable family of Borel sets separating the points of ${(\mathbb{T}^2)}^*$, and let 
		\begin{displaymath}
			B_n' =
			\begin{cases}
				B_n, &\text{if~~} m(B_n) = 1 \\
				{(\mathbb{T}^2)}^* \setminus B_n,  &\text{if ~~} m(B_n) = 0,
			\end{cases}
		\end{displaymath}
		chosen so that $m(C\cap\pi^{-1}(B_n'))=1$. Then $\bigcap_n\bigl(C\cap\pi^{-1}(B_n')\bigr)$ has full $m$-measure, and $\pi$ is constant on it, since $\bigcap_nB_n'$ contains at most one point. Hence a set of full $m$-measure is contained in a single orbit of $R_\theta$, which is countable. This contradicts the fact that $m$ has no atoms. Foliations of this type also arise as the stable and unstable foliations of hyperbolic toral automorphisms.
	\end{remark}
	
	\
	
	For each $x \in X$, let $\mathcal L_x$ be the leaf of $\mathcal F$ which contains $x \in X$. The foliation provides a notion of points which belong to the same geometric layer of the space. We shall regard these leaves as the allowed regions along which mass may be transported. 
	
	\
	
	\begin{definition}[Foliated transport plan]
		Let $X$ be a Riemannian manifold with associated distance $d$, endowed with a Borel foliation $\mathcal{F}$. Given $\mu$ and $\nu$ probability measures on $X$, we say that a transport plan $\gamma \in \Pi(\mu, \nu)$ is a foliated transport plan if $\gamma \bigl( \{ (x,x') \in X \times X : x' \notin \mathcal{L}_x \} \bigr) = 0$, or, equivalently, if $\gamma (\bigl\{(x,x')\in X\times X:\pi(x)=\pi(x')\bigr\}) = 1$. Moreover, we define 
		\begin{equation*}
			\Pi_{\mathcal{F}}(\mu,\nu) := \left\{ \gamma \in \Pi (\mu, \nu) : ~ \gamma \left( \bigcup_{y\in X^*} (\mathcal{L}_y\times \mathcal{L}_y) \right) = 1
			\right\}
		\end{equation*}
		the set of foliated transport plans between $\mu$ and $\nu$ given $\mathcal{F}$.
	\end{definition} 
	
	\
	
	\noindent The set appearing in this definition is measurable. Indeed, consider
	\begin{displaymath}
		E_{\mathcal F}:=\bigcup_{y\in X^*}(\mathcal L_y\times\mathcal L_y)=\bigl\{(x,x')\in X\times X:\pi(x)=\pi(x')\bigr\}=(\pi\times\pi)^{-1}(\Delta_{X^*}),
	\end{displaymath}
	where $\Delta_{X^*}:=\{(y,y):y\in X^*\}$. Since $X^*$ is a standard Borel space, it admits a countable family $\{B_n\}_{n\in\N}$ of Borel sets separating its points, and then $\Delta_{X^*}=\bigcap_{n\in\N}\bigl[(B_n\times B_n)\cup(B_n^c\times B_n^c)\bigr]$ is a Borel subset of $X^*\times X^*$. Since $\pi\times\pi$ is Borel, $E_{\mathcal F}$ is a Borel subset of $X\times X$. We keep the notation $E_{\mathcal F}$ throughout the paper.
		
	In this context, the foliation acts as a geometric constraint on transport problems: mass may move along a leaf, but it cannot be transported from one leaf to another.
	
	\
	
	\begin{proposition} \label{Propositon_foliated_plans}
		Let $X$ be a Riemannian manifold with associated distance $d$, endowed with a Borel foliation $\mathcal{F}$. Given $\mu$ and $\nu$ probability measures on $X$, we have $\Pi_{\mathcal F}(\mu,\nu) \neq \varnothing$ if and only if $\pi_* \mu = \pi_* \nu$, where $\pi$ is the Borel map $\pi: X \to X^*$.
	\end{proposition} 
	
	\begin{proof}
			
		Assume that $\Pi_{\mathcal F}(\mu,\nu)\neq\varnothing$. Let $\gamma \in \Pi_{\mathcal F}(\mu,\nu)$. Then, by definition of foliated transport plan, for $\gamma$-almost every $(x,x')\in X\times X$, there exists some $y\in X^*$ such that $x \in \mathcal{L}_y$ and $x'\in \mathcal{L}_y$. Moreover, note that if $x, x' \in \mathcal{L}_y$, then $\pi(x) = \pi(x')$. Hence $\pi(x) = \pi(x')$ for $\gamma$-almost every $(x,x')\in X \times X$. Now, let $\varphi:X^* \to \mathbb R$ be any bounded measurable function. Note that
		\begin{equation*}
			\int_{X^*} \varphi(y) ~d((\pi \circ \operatorname{proj}_1)_{*} \gamma)(y) = \int_{X \times X} \varphi(\pi(x)) ~d\gamma(x,x'),
		\end{equation*}
		and
		\begin{equation*}
			\int_{X^*} \varphi(y) ~d((\pi \circ \operatorname{proj}_2)_{*} \gamma)(y) =
			\int_{X \times X} \varphi (\pi(x')) ~d \gamma(x,x').
		\end{equation*}
				
		Since $\pi(x) = \pi(x')$ for $\gamma$-almost every $(x,x')$, we have $\varphi(\pi(x)) = \varphi(\pi(x'))$ for $\gamma$-almost every $(x,x')$, and therefore
		\begin{displaymath}
			\int_{X\times X}\varphi(\pi(x)) ~d\gamma =	\int_{X \times X} \varphi(\pi(x')) ~d\gamma.
		\end{displaymath}
		Thus,
		\begin{displaymath}
			\int_{X^*} \varphi(y) ~d((\pi \circ \operatorname{proj}_1)_{*} \gamma)(y) = \int_{X^*} \varphi(y) ~d((\pi \circ \operatorname{proj}_2)_{*} \gamma)(y).
		\end{displaymath}
		But $(\pi \circ \operatorname{proj}_1)_{*} \gamma = \pi_{*} \mu$ and  $(\pi \circ \operatorname{proj}_2)_{*} \gamma = \pi_{*} \nu$. Hence $\pi_{*}\mu = \pi_{*}\nu$.
		
		\	
		
		Now assume that $\pi_{*}\mu = \pi_{*}\nu = \sigma$. By Theorem \ref{theorem:disintegration}, there exist families of probability measures $\{ \mu^y\}_{y \in X^*}$ and $\{\nu^y\}_{y \in X^*}$ on $X$ such that, for every Borel set $A\subset X$, the maps $y \mapsto \mu^{y}(A)$ and $y \mapsto \nu^{y}(A)$ are Borel and
		\begin{displaymath}
			\mu(A) = \int_{X^*} \mu^y(A) ~d\sigma(y),
			\qquad
			\nu(A) = \int_{X^*} \nu^y(A) ~d\sigma(y),
		\end{displaymath}
		and such that $\mu^y$ and $\nu^y$ are concentrated on $\mathcal{L}_y=\pi^{-1}(\{y\})$ for $\sigma$-almost every $y\in X^*$. Define $\eta^y := \mu^y \otimes \nu^y$. The map $y\mapsto\eta^y(C)$ is Borel for every Borel set $C\subset X\times X$. Indeed, for a Borel rectangle $A\times B$ we have $\eta^y(A\times B)=\mu^y(A)\,\nu^y(B)$, which is Borel in $y$. The class of Borel sets $C\subset X\times X$ for which $y\mapsto\eta^y(C)$ is Borel contains $X\times X$ and is closed under proper differences and increasing countable unions. Since it contains the $\pi$-system of Borel rectangles, Dynkin's $\pi$-$\lambda$ theorem shows that it contains the $\sigma$-algebra generated by the rectangles, which is $\mathcal B(X\times X)$ because $X$ is separable and metrizable. Hence, by monotone convergence,
		\begin{displaymath}
			\eta(C) := \int_{X^*} \eta^y(C) ~d\sigma(y), \qquad C\in\mathcal B(X\times X),
		\end{displaymath}
		defines a probability measure on $X\times X$. Now, we claim that $\eta \in \Pi_{\mathcal F}(\mu,\nu)$. Indeed, let $\varphi: X \to \mathbb R$ be bounded and measurable. Then,
		\begin{align*}
			\int_X \varphi(x) ~d((\operatorname{proj}_1)_*\eta) (x) =& \int_{X \times X} \varphi(x) ~d\eta (x, x') \\
			=& \int_{X^*} \left( \int_{X \times X} \varphi(x) ~d\eta^y(x,x') \right) d\sigma (y).
		\end{align*}
		Since $\eta^y \in \Pi(\mu^y, \nu^y)$, we have
		\begin{displaymath}
			\int_{X \times X} \varphi(x) ~d\eta^y = \int_X \varphi(x) ~d\mu^y. 
		\end{displaymath}
		Therefore,
		\begin{displaymath}
			\int_X \varphi(x) ~d((\operatorname{proj}_1)_*\eta) (x) = \int_{X^*} \left( \int_X \varphi(x) ~d\mu^y \right) d\sigma (y) = \int_X \varphi (x) ~d\mu (x).
		\end{displaymath}
		Thus, we conclude that $(\operatorname{proj}_1)_*\eta = \mu$ and, analogously $(\operatorname{proj}_2)_*\eta = \nu$. Recall that $E_{\mathcal F} = \bigcup_{y\in X^*}(\mathcal{L}_y\times \mathcal{L}_y)$ is a Borel set. Since both $\mu^y$ and $\nu^y$ are concentrated on $\mathcal{L}_y$, we have $\eta^y(\mathcal{L}_y \times \mathcal{L}_y) = \mu^y(\mathcal{L}_y)\nu^y(\mathcal{L}_y) = 1$ for $\sigma$-almost every $y$. 	Note that $E_{\mathcal F}^c \subseteq (\mathcal{L}_y \times \mathcal{L}_y)^c$, and then $\eta^y (E_{\mathcal F}^c) \leq \eta^y \big((\mathcal{L}_y \times \mathcal{L}_y)^c\big) = 0$. Therefore, $\eta^y(E_{\mathcal F}^c) = 0$ for $\sigma$-almost every $y$, and
		\begin{displaymath}
			\eta (E_{\mathcal F}^c) = \int_{X^*} \eta^y(E_{\mathcal F}^c) ~d\sigma(y) = 0.
		\end{displaymath}
		Consequently, $\eta(E_{\mathcal F}) = 1$. Therefore $\eta \in \Pi_ {\mathcal{F}} (\mu, \nu)$.
	\end{proof}

	\
	
	\begin{remark} \label{remark:saturated-sets}
		The condition $\pi_*\mu=\pi_*\nu$ in Proposition \ref{Propositon_foliated_plans} can be expressed without reference to $\pi$. Indeed, a set $A\subset X$ is said to be saturated if it is a union of leaves of $\mathcal F$. On the one hand, for every Borel set $B\subset X^*$, the set $\pi^{-1}(B)$ is a saturated Borel set. On the other hand, let $A\subset X$ be a saturated Borel set. Then $\pi(A)$ and $\pi(X\setminus A)$ are analytic subsets of $X^*$, and they are disjoint because $A$ is saturated. By the Lusin Separation Theorem \cite[Theorem 14.7]{Kec95}, there exists a Borel set $B\subset X^*$ with $\pi(A)\subset B$ and $B\cap\pi(X\setminus A)=\varnothing$, and then $A=\pi^{-1}(B)$. Consequently, $\pi_*\mu=\pi_*\nu$ if and only if $\mu(A)=\nu(A)$ for every saturated Borel set $A\subset X$.
	\end{remark}
	
	\
	
	\begin{proposition} \label{proposition_2}
		Let $X$ be a Riemannian manifold endowed with a Borel foliation $\mathcal{F}$, and $\mu$ and $\nu$ probability measures on $X$. If $\gamma \in \Pi_{\mathcal{F}}(\mu, \nu)$, then $\pi_*\mu=\pi_*\nu=:\sigma$, and there exist a disintegration $\{ \gamma^y \}_{y\in X^*}$ of $\gamma$ with respect to $\pi\circ\operatorname{proj}_1$ and disintegrations $\{\mu^y\}_{y\in X^*}$ and $\{\nu^y\}_{y\in X^*}$ of $\mu$ and $\nu$, respectively, with respect to $\pi$, such that, for $\sigma$-almost every $y \in X^*$, the measure $\gamma^y$ is concentrated on $\mathcal{L}_ y \times \mathcal{L}_ y$, $(\operatorname{proj}_1)_* \gamma^y = \mu^y$ and $(\operatorname{proj}_2)_* \gamma^y = \nu^y$.
	\end{proposition}
	
	\begin{proof}
		Since $\gamma \in \Pi_{\mathcal F}(\mu,\nu)$, by Proposition \ref{Propositon_foliated_plans} we have $\pi_*\mu=\pi_*\nu$. Define $\sigma:=\pi_*\mu=\pi_*\nu$. Now, we apply the disintegration theorem to $\gamma$ with respect	to the measurable map $q:=\pi\circ\operatorname{proj}_1$, where $\pi: X \to X^*$ is the Borel map of Definition \ref{Def-Borel-foliation}. Since $q_*\gamma=\pi_*\mu=\sigma$, there exists a family of probability measures
		$\{\gamma^y\}_{y\in X^*}$ on $X\times X$, with $y\mapsto\gamma^y(C)$ Borel for every Borel set $C\subset X\times X$, such that
		\begin{displaymath}
			\gamma	=	\int_{X^*}\gamma^y\,d\sigma(y)
		\end{displaymath}
		and $\gamma^y\left(q^{-1}(\{y\})\right)=1$ for $\sigma$-almost every $y\in X^*$.
		
		\
		
		\noindent Recall that $E_{\mathcal F} = \bigcup_{y\in X^*} (\mathcal{L}_y\times \mathcal{L}_y)$. Since $\gamma (E_{\mathcal F}) =1$, we have $\gamma^y(E_{\mathcal F})=1$ for $\sigma$-almost every $y\in X^*$. Moreover, for every $y\in X^*$,
		\begin{displaymath}
			q^{-1}(\{y\}) = \left\{(x,x')\in X\times X:\pi(x)=y\right\} = \mathcal L_y\times X,
		\end{displaymath}
		and hence $q^{-1}(\{y\})\cap E_{\mathcal F} = \mathcal L_y\times\mathcal L_y$.
		Thus, $\gamma^y\left(\mathcal L_y\times\mathcal L_y\right)=1$ for $\sigma$-almost every $y\in X^*$. In other words, $\gamma^y$ is concentrated on $\mathcal L_y\times\mathcal L_y$.
		
		\
		
		\noindent Define $\widetilde{\mu}^y := (\operatorname{proj}_1)_*\gamma^y$ and $\widetilde{\nu}^y := (\operatorname{proj}_2)_*\gamma^y$. The family $\{\widetilde{\mu}^y\}_{y\in X^*}$ is a disintegration
		of $\mu$ with respect to $\pi$. Indeed, for every bounded Borel function $\varphi:X\to\mathbb R$, the	disintegration formula and the definition of the pushforward give
		\begin{displaymath}
			\begin{split}
				\int_X\varphi(x)\,d\mu(x)
				&=
				\int_{X\times X}
				\varphi\circ\operatorname{proj}_1(x,x')\,d\gamma(x,x')\\
				&=
				\int_{X^*}
				\left(
				\int_{X\times X}
				\varphi\circ\operatorname{proj}_1(x,x')\,d\gamma^y(x,x')
				\right)
				d\sigma(y)\\
				&=
				\int_{X^*}
				\left(
				\int_X
				\varphi(x)\,d\left((\operatorname{proj}_1)_*\gamma^y\right)(x)
				\right)
				d\sigma(y)\\
				&=
				\int_{X^*}
				\left(
				\int_X\varphi(x)\,d\widetilde{\mu}^y(x)
				\right)
				d\sigma(y).
			\end{split}
		\end{displaymath}
		Therefore,
		\begin{displaymath}
			\mu=\int_{X^*}\widetilde{\mu}^y\,d\sigma(y).
		\end{displaymath}
		Moreover, $\widetilde{\mu}^y$ is concentrated on $\mathcal L_y$
		for $\sigma$-almost every $y$, and $y\longmapsto\widetilde{\mu}^y(A)$ is Borel for every $A\in\mathcal B(X)$, since $\operatorname{proj}_1^{-1}(A)$ belongs to $\mathcal B(X\times X)$ for every $A\in\mathcal B(X)$, and $\{\gamma^y\}_{y \in X^*}$ is a disintegration. Similarly, $\{\widetilde{\nu}^y\}_{y\in X^*}$ is a
		disintegration of $\nu$ with respect to $\pi$.
		
		Setting $\mu^y:=\widetilde{\mu}^y$ and $\nu^y:=\widetilde{\nu}^y$, we obtain disintegrations of $\mu$ and $\nu$ with respect to $\pi$ such that $(\operatorname{proj}_1)_*\gamma^y=\mu^y$ and $(\operatorname{proj}_2)_*\gamma^y=\nu^y$ for every $y\in X^*$. By the essential uniqueness of disintegrations, any other disintegrations of $\mu$ and $\nu$ with respect to $\pi$ coincide with these for $\sigma$-almost every $y\in X^*$.
	\end{proof}
	
	\
	
	Proposition \ref{proposition_2} shows that every foliated transport plan can be decomposed into a family of leafwise transport plans. Thus, a foliated transport plan may be regarded as a family of transport plans acting independently along the leaves,	with the common transverse measure $\sigma$ describing how the total mass is distributed among them. In particular, the condition $\pi_*\mu=\pi_*\nu$ expresses the conservation of mass at the transverse level, while each $\gamma^y$ describes the transportation of the conditional measure $\mu^y$ into the conditional measure $\nu^y$ within the leaf $\mathcal L_y$.
	
	\
	
	For $x,x'\in X$ belonging to the same leaf, let	$d_{\mathcal{L}}(x,x')$ denote the intrinsic Riemannian distance
	between $x$ and $x'$ measured within that leaf. More precisely, if
	$x,x'\in\mathcal L_y$, we set
	\begin{displaymath}
		d_{\mathcal L}(x,x') := \inf \left\{ \int_0^1\|\dot\alpha(t)\|_g\,dt: ~\alpha\in\operatorname{AC}_{x,x'}(\mathcal L_y)\right\},
	\end{displaymath}
	where $\operatorname{AC}_{x,x'}(\mathcal L_y)$ denotes the set of absolutely continuous curves $\alpha:[0,1]\to X$ with $\alpha([0,1])\subset\mathcal L_y$, $\alpha(0)=x$ and $\alpha(1)=x'$. If $x$ and $x'$ belong to different leaves, we set $d_{\mathcal L}(x,x'):=+\infty$.
	
	The curves in this definition are taken in $X$, but nothing changes if they are taken in the leaf with its own manifold structure. Indeed, a leaf meets each foliated chart $U$ in at most countably many plaques. If $\alpha$ is a continuous curve in $X$ with image in a leaf, then on each connected component $J$ of $\alpha^{-1}(U)$ the transverse coordinate $\operatorname{pr}_2\circ\phi\circ\alpha$ is continuous with countable range, hence constant, so that $\alpha(J)$ lies in a single plaque, which is an embedded $C^1$ submanifold of $U$. Consequently, $\alpha$ is continuous for the leaf topology, and it is absolutely continuous as a curve in $X$ if and only if it is absolutely continuous as a curve in the leaf, with the same velocity. In particular, the restriction of $d_{\mathcal L}$ to each leaf is the distance induced by the Riemannian metric of that leaf.
	
	The extended leafwise distance is always Borel, and it is lower semicontinuous when $X$ is complete. To prove this, we use the length functional, defined for $\omega\in C([0,1];X)$ by
	\begin{displaymath}
		\operatorname{Length}(\omega):=\sup\Bigl\{\sum_{i=1}^Nd\bigl(\omega(t_{i-1}),\omega(t_i)\bigr):0=t_0<t_1<\dots<t_N=1\Bigr\}.
	\end{displaymath}
	It is a supremum of continuous functions on $C([0,1];X)$, hence lower semicontinuous with respect to uniform convergence. For absolutely continuous curves $\omega$, it coincides with $\int_0^1\|\dot\omega(t)\|_g\,dt$.
	
	\
	
	\begin{lemma} \label{lemma:dL-Borel-lsc}
		Let $X$ be a Riemannian manifold endowed with a foliation $\mathcal F$ of class $C^1$. Then, one has
		\begin{enumerate}[leftmargin=*]
			\item[(i)] The map $d_{\mathcal L}:X\times X\to[0,+\infty]$ is Borel, and the set $\{d_{\mathcal L}<+\infty\}$ of pairs of points lying in the same leaf is an $F_\sigma$ subset of $X\times X$.
			\item[(ii)] If $X$ is complete, then $d_{\mathcal L}$ is lower semicontinuous, i.e., the set $\{d_{\mathcal L}\leq r\}$ is closed in $X\times X$ for every $r\geq0$.
		\end{enumerate}
	\end{lemma}
	
	\begin{proof}
		For a compact set $K\subset X$ and $x,x'\in X$, let $d_K(x,x')$ be the infimum of $\int_0^1\|\dot\alpha(t)\|_g\,dt$ over all absolutely continuous curves $\alpha:[0,1]\to K$ with image in a single leaf, $\alpha(0)=x$ and $\alpha(1)=x'$, with $d_K(x,x')=+\infty$ if there is no such curve. We first show that the set $\{d_K\leq r\}$ is closed for every nonnegative $r$. Let $(x_n,x_n')\to(x,x')$ with $d_K(x_n,x_n')\leq r$. Choose absolutely continuous curves $c_n:[0,1]\to K$ from $x_n$ to $x_n'$, each with image in a single leaf and of length at most $r+1/n$, parametrized proportionally to arc length. Since the distance $d$ between two points is bounded above by the length of any curve joining them, the curves $c_n$ are $(r+1)$-Lipschitz with respect to $d$. By the Arzel\`a-Ascoli theorem, a subsequence, still denoted by $(c_n)$, converges uniformly to an $(r+1)$-Lipschitz curve $c:[0,1]\to K$ with $c(0)=x$ and $c(1)=x'$. Cover $[0,1]$ by finitely many closed intervals $J$, consecutive ones overlapping, such that each $c(J)$ is contained in the domain $U$ of a foliated chart $\phi$. Fix such an interval $J$. Since $c(J)$ is compact, $c_n(J)\subset U$ for $n$ large, and then, by the discussion following the definition of $d_{\mathcal L}$, the transverse coordinate $\operatorname{pr}_2\circ\phi\circ c_n$ is constant on $J$. Its uniform limit $\operatorname{pr}_2\circ\phi\circ c$ is therefore constant on $J$, so $c(J)$ lies in a single plaque. Since consecutive intervals overlap, the image of $c$ lies in a single leaf. Moreover, $c$ is absolutely continuous, being Lipschitz. By the lower semicontinuity of the length functional,
		\begin{displaymath}
			d_K(x,x')\leq ~ \operatorname{Length}(c) ~ \leq ~ \liminf_{n\to\infty} ~ \operatorname{Length}(c_n) ~ \leq r.
		\end{displaymath}
		
		\noindent (i) Let $K_1\subset K_2\subset\cdots$ be compact subsets of $X$ with $K_j\subset\operatorname{int}K_{j+1}$ and $\bigcup_jK_j=X$. The image of every curve in the definition of $d_{\mathcal L}$ is compact, hence contained in some $K_j$. Therefore $d_{\mathcal L}=\inf_jd_{K_j}$, and, for every $r>0$,
		\begin{displaymath}
			\{d_{\mathcal L}<r\}=\bigcup_{j\in\N}\bigcup_{m\in\N}\bigl\{d_{K_j}\leq r-1/m\bigr\}
		\end{displaymath}
		is an $F_\sigma$ set. Hence $d_{\mathcal L}$ is Borel, and $\{d_{\mathcal L}<+\infty\}=\bigcup_{n\in\N}\{d_{\mathcal L}<n\}$ is an $F_\sigma$ set.
		
		\
		
		\noindent (ii) Assume that $X$ is complete, so that closed balls in $X$ are compact by the Hopf-Rinow theorem. Let $(x_n,x_n')\to(x,x')$ with $d_{\mathcal L}(x_n,x_n')\leq r$, and let $K$ be the closed ball of radius $r+2$ centered at $x$. For $n$ large, $d(x_n,x)<1$, so every curve starting at $x_n$ with length at most $r+1$ is contained in $K$. Hence $d_K(x_n,x_n')\leq r$ for $n$ large, and the closedness of $\{d_K\leq r\}$ gives $d_{\mathcal L}(x,x')\leq d_K(x,x')\leq r$.
	\end{proof}

	\
	
	\begin{remark} \label{remark:dL-completeness}
		Lemma \ref{lemma:dL-Borel-lsc} does not use Definition \ref{Def-Borel-foliation}, so it also applies to foliations that are not Borel, such as the linear foliation of Remark \ref{remark:scope-Borel-foliation}. For a Borel foliation, $\{d_{\mathcal L}<+\infty\}=E_{\mathcal F}$, since leaves are connected. Completeness cannot be dropped in item (ii). Indeed, for the foliation of $X=\mathbb R^2\setminus\{0\}$ by horizontal lines, $d_{\mathcal L}\bigl((-1,1/n),(1,1/n)\bigr)=2$ for every $n$, whereas $d_{\mathcal L}\bigl((-1,0),(1,0)\bigr)=+\infty$, because $(-1,0)$ and $(1,0)$ lie in different leaves.
	\end{remark}

	\
	
	In particular, the integrals of $d_{\mathcal L}^2$ against Borel probability measures on $X\times X$ appearing below are well defined.
	
	\
	
	\begin{definition}[Foliated quadratic total cost]
		Let $X$ be a Riemannian manifold endowed with a Borel foliation $\mathcal{F}$, and $\mu, \nu \in \mathscr{P}(X)$. The foliated quadratic transport cost between $\mu$ and $\nu$ is defined by
		\begin{displaymath}
			W_{2,\mathcal F}^2(\mu,\nu)
			:=
			\inf_{\gamma\in\Pi_{\mathcal F}(\mu,\nu)}
			~ \int_{X\times X}
			d_{\mathcal L}^2 (x,x') ~d\gamma(x,x').
		\end{displaymath}
		As usual, we adopt the convention that $\inf\varnothing=+\infty$.
	\end{definition}
	
	\
	
	By Proposition \ref{Propositon_foliated_plans}, the set $\Pi_{\mathcal F}(\mu,\nu)$ is nonempty precisely when the two measures have the same transverse projection, that is, $\pi_*\mu=\pi_*\nu$. Hence, $W_{2,\mathcal F}^2 (\mu,\nu)=+\infty$,	whenever $\pi_*\mu\neq\pi_*\nu$. When $\pi_*\mu=\pi_*\nu$, the foliated cost admits a leafwise decomposition. This is Theorem \ref{Theorem-2} below, which together with Proposition \ref{Propositon_foliated_plans} and Corollary \ref{corollary:existence-optimal-plans} proves Theorem \ref{maintheorem:A}. To prove it, we need the following lemma on the measurable selection of leafwise transport plans.
	
	\
	
	\begin{lemma} \label{lemma-measurable-leafwise-selection}
		
		Let $X$ be a Riemannian manifold endowed with a Borel foliation $\mathcal F$, and let $\mu$ and $\nu\in\mathscr P(X)$. Given $\gamma \in \Pi_{\mathcal{F}}(\mu, \nu)$, consider $\sigma := \pi_* \mu = \pi_* \nu$, and disintegrations $\{\mu^y\}$ and $\{\nu^y\}$ of $\mu$ and $\nu$, respectively, according to Proposition \ref{proposition_2}. Let $c:X\times X\longrightarrow[0,+\infty]$	be a nonnegative Borel function, and define
		\begin{displaymath}
			m(y)
			:=
			\inf_{\eta \in\Pi(\mu^y,\nu^y)}
			\int_{X\times X}c(x,x') ~d\eta(x,x').
		\end{displaymath}
		Then $m$ is measurable with	respect to the completion of $\sigma$. Moreover, for every $\varepsilon>0$, there exists a family $\{\gamma_\varepsilon^y\}_{y\in X^*}$, measurable with respect to the completion of $\sigma$, such that $\gamma_\varepsilon^y\in\Pi(\mu^y,\nu^y)$ and 
		\begin{displaymath}
			\int_{X\times X}c(x,x') ~d\gamma_\varepsilon^y(x,x')
			\leq
			m(y)+\varepsilon
		\end{displaymath}
		for $\sigma$-almost every $y\in X^*$. In addition, if
		\begin{displaymath}
			\gamma_\varepsilon	:= \int_{X^*}\gamma_\varepsilon^y\,d\sigma(y),
		\end{displaymath}
		then $\gamma_\varepsilon \in \Pi_{\mathcal{F}}(\mu, \nu)$.
	\end{lemma}
	
	\begin{proof}
		First, note that since $X$ is a standard Borel space, the spaces	$\mathscr P(X)$ and $\mathscr P(X\times X)$, endowed with their	canonical Borel structures, are also standard Borel spaces \cite[Section 17.E]{Kec95}. 
		
		\
		
		\noindent Consider $\Gamma(y):=\Pi(\mu^y,\nu^y) \subset\mathscr P(X\times X)$. Note that the maps $\eta \mapsto (\operatorname{proj}_1)_*\eta$ and $\eta \mapsto (\operatorname{proj}_2)_*\eta$, for $\eta \in \mathscr{P}( X \times X)$ are Borel. Moreover, the maps $	y \mapsto \mu^y$ and $y \mapsto\nu^y$	are Borel, since $y\mapsto\mu^y(A)$ and $y\mapsto\nu^y(A)$ are Borel for every Borel set $A\subset X$, by Proposition \ref{proposition_2}, and the Borel structure of $\mathscr P(X)$ is generated by the maps $\lambda\mapsto\lambda(A)$ \cite[Theorem 17.24]{Kec95}. Consequently, the graph of $\Gamma$, $\operatorname{Gr}(\Gamma) :=	\bigl\{(y,\eta)\in X^*\times\mathscr P(X\times X):
		\eta\in\Pi(\mu^y,\nu^y)\bigr\}$ is Borel.
		
		\
		
		\noindent Define the Borel measurable map $I:\mathscr P(X\times X)\longrightarrow[0,+\infty]$ by
		\begin{displaymath}
			I(\eta)	:=	\int_{X\times X}c(x,x')\,d\eta(x,x').
		\end{displaymath}
		Then, we have $m(y) = \inf_{\eta\in\Gamma(y)}I(\eta)$. By \cite[Proposition 7.47]{BS78}, $m$ is lower semianalytic. Since every analytic subset	of a standard Borel space is universally measurable, every lower	semianalytic function is universally measurable. In particular,
		$m$ is measurable with respect to the completion of $\sigma$. Moreover, for every
		$\varepsilon>0$, by \cite[Proposition 7.50]{BS78} there exists a universally measurable selection $y\mapsto\gamma_\varepsilon^y \in \Gamma(y)$ satisfying $I(\gamma_\varepsilon^y)\leq m(y)+\varepsilon$, whenever $m(y)<+\infty$. For every Borel set $A\subset X\times X$, the map $y\mapsto\gamma_\varepsilon^y(A)$ is universally measurable, being the composition of this selection with the Borel map $\eta\mapsto\eta(A)$, and hence it is measurable with respect to the completion of $\sigma$. Define
		\begin{displaymath}
			\gamma_\varepsilon(A) := \int_{X^*}\gamma_\varepsilon^y(A)\,d\sigma(y)
		\end{displaymath}
		for every Borel set $A\subset X\times X$. For every bounded Borel function	$f:X\to\mathbb R$, we have
		\begin{align*}
			\int_X f(x) \,d(\operatorname{proj}_1)_*\gamma_\varepsilon(x)
			&=
			\int_{X\times X}f(x)\,d\gamma_\varepsilon(x,x')\\
			&=
			\int_{X^*}
			\left(
			\int_{X\times X}f(x)\,d\gamma_\varepsilon^y(x,x')
			\right)d\sigma(y)\\
			&=
			\int_{X^*}
			\left(
			\int_Xf(x)\,d\mu^y(x)
			\right)d\sigma(y).
		\end{align*}
		Therefore,
		\begin{displaymath}
			(\operatorname{proj}_1)_*\gamma_\varepsilon	=
			\int_{X^*}\mu^y\,d\sigma(y).
		\end{displaymath}
		The same argument gives
		\begin{displaymath}
			(\operatorname{proj}_2)_*\gamma_\varepsilon
			=
			\int_{X^*}\nu^y\,d\sigma(y).
		\end{displaymath}
		For $\sigma$-almost every $y$, both marginals of $\gamma_\varepsilon^y$ are concentrated on $\pi^{-1}(\{y\})$. Hence, $\gamma_\varepsilon^y \bigl(\pi^{-1}(\{y\})\times\pi^{-1}(\{y\}) \bigr) =	1$, and then $	\gamma_\varepsilon^y(\bigl\{(x,x')\in X\times X:\pi(x)=\pi(x')\bigr\})=1$	for $\sigma$-almost every $y$. Therefore, $\gamma_\varepsilon \in \Pi_{\mathcal{F}} (\mu, \nu)$.
	\end{proof}
	
	\
	
	\begin{theorem}\label{Theorem-2}
		Let $X$ be a Riemannian manifold endowed with a Borel foliation $\mathcal F$, and let $\mu,\nu\in\mathscr P(X)$ satisfy $\pi_*\mu=\pi_*\nu=:\sigma$. Let $\{\mu^y\}_{y\in X^*}$ and $\{\nu^y\}_{y\in X^*}$ be disintegrations of $\mu$ and $\nu$, respectively, with respect to $\pi$. Then
		\begin{displaymath}
			W_{2,\mathcal F}^{2}(\mu,\nu)
			=
			\int_{X^*}
			W_{2,\mathcal L_y}^{2}(\mu^y,\nu^y)\,d\sigma(y),
		\end{displaymath}
		where
		\begin{displaymath}
			W_{2,\mathcal L_y}^{2}(\mu^y,\nu^y)
			:=
			\inf_{\eta \in\Pi(\mu^y,\nu^y)}
			\int_{\mathcal L_y\times\mathcal L_y}
			d_{\mathcal L}^2(x,x') ~d\eta(x,x').
		\end{displaymath}
	\end{theorem}

	\
	
	Here the integrand is measurable with respect to the completion of $\sigma$, by Lemma \ref{lemma-measurable-leafwise-selection} and the essential uniqueness of disintegrations, and both sides may be equal to $+\infty$.
	
	\
	
	\begin{proof}
		By Proposition \ref{Propositon_foliated_plans}, the set $\Pi_{\mathcal F}(\mu,\nu)$ is nonempty. Let $\gamma\in\Pi_{\mathcal F}(\mu,\nu)$. By Proposition \ref{proposition_2} and the essential uniqueness of disintegrations, there is a disintegration $\{\gamma^y\}_{y\in X^*}$ of $\gamma$ with respect to $\pi\circ\operatorname{proj}_1$ such that the marginals of $\gamma^y$ are $\mu^y$ and $\nu^y$ for $\sigma$-almost every $y$. Therefore, we have $\gamma^y\in\Pi(\mu^y,\nu^y)$ for $\sigma$-almost every $y$. Moreover,
		\begin{displaymath}
			\int_{X\times X}
			d_{\mathcal L}^2(x,x') ~d\gamma(x,x')
			=
			\int_{X^*}
			\left(
			\int_{\mathcal L_y\times\mathcal L_y}
			d_{\mathcal L}^2 (x,x') ~d\gamma^y(x,x')
			\right)d\sigma(y).
		\end{displaymath}
		
		\
		
		Since $\gamma^y\in\Pi(\mu^y,\nu^y)$, we have
		\begin{displaymath}
			\int_{\mathcal L_y\times\mathcal L_y}
			d_{\mathcal L}^{2}(x,x')\,d\gamma^y(x,x')
			\geq
			W_{2,\mathcal L_y}^{2}(\mu^y,\nu^y)
		\end{displaymath}
		for $\sigma$-almost every $y$. Hence,
		\begin{displaymath}
			\int_{X\times X} d_{\mathcal L}^{2}(x, x') \,d\gamma(x, x')
			\geq
			\int_{X^*}
			W_{2,\mathcal L_y}^{2}(\mu^y,\nu^y)\,d\sigma(y).
		\end{displaymath}
		
		Taking the infimum over
		$\gamma\in\Pi_{\mathcal F}(\mu,\nu)$ gives
		\begin{equation}
			W_{2,\mathcal F}^{2}(\mu,\nu)
			\geq
			\int_{X^*}
			W_{2,\mathcal L_y}^{2}(\mu^y,\nu^y)\,d\sigma(y).
			\label{eq:first-leafwise-inequality}
		\end{equation}
		
		In particular, if $W_{2,\mathcal F}^{2}(\mu,\nu)<+\infty$, then the leafwise cost is finite for $\sigma$-almost every $y$, that is, $W_{2,\mathcal L_y}^{2}(\mu^y,\nu^y)<+\infty$, and it is integrable with respect to $\sigma$.
		
		\
		
		We now prove the reverse inequality. If the right-hand side of \eqref{eq:first-leafwise-inequality} is $+\infty$, there is nothing to prove, so we may assume that it is finite. Fix $\varepsilon>0$. By Lemma \ref{lemma-measurable-leafwise-selection} and the essential uniqueness of disintegrations, there exists a
		family $\{\gamma_\varepsilon^y\}_{y\in X^*}$, measurable with respect to the completion of $\sigma$, such that $\gamma_\varepsilon^y\in\Pi(\mu^y,\nu^y)$ and
		\begin{displaymath}
			\int_{\mathcal L_y\times\mathcal L_y}
			d_{\mathcal L}^{2}(x,x')\,d\gamma_\varepsilon^y(x,x')
			\leq
			W_{2,\mathcal L_y}^{2}(\mu^y,\nu^y)+\varepsilon
		\end{displaymath}
		for $\sigma$-almost every $y$. Moreover, we have $\gamma_\varepsilon \in \Pi_{\mathcal{F}} (\mu, \nu)$, where
		\begin{displaymath}
			\gamma_\varepsilon
			:=
			\int_{X^*}\gamma_\varepsilon^y\,d\sigma(y).
		\end{displaymath}

		It follows that
		\begin{align*}
			W_{2,\mathcal F}^{2}(\mu,\nu)
			&\leq
			\int_{X\times X}
			d_{\mathcal L}^{2}(x,x')\,d\gamma_\varepsilon(x,x')\\
			&=
			\int_{X^*}
			\left(
			\int_{\mathcal L_y\times\mathcal L_y}
			d_{\mathcal L}^{2}(x,x')\,d\gamma_\varepsilon^y(x,x')
			\right)d\sigma(y)\\
			&\leq
			\int_{X^*}
			W_{2,\mathcal L_y}^{2}(\mu^y,\nu^y)\,d\sigma(y)
			+\varepsilon.
		\end{align*}
		
		Since $\varepsilon>0$ is arbitrary, we conclude that
		\begin{equation}
			W_{2,\mathcal F}^{2}(\mu,\nu)
			\leq
			\int_{X^*}
			W_{2,\mathcal L_y}^{2}(\mu^y,\nu^y)\,d\sigma(y).
			\label{eq:second-leafwise-inequality}
		\end{equation}
		
		Combining
		\eqref{eq:first-leafwise-inequality} and
		\eqref{eq:second-leafwise-inequality}, we obtain
		\begin{displaymath}
			W_{2,\mathcal F}^{2}(\mu,\nu)
			=
			\int_{X^*}
			W_{2,\mathcal L_y}^{2}(\mu^y,\nu^y)\,d\sigma(y).
		\end{displaymath}
	\end{proof}
	
	\
	
	Theorem \ref{Theorem-2} shows that the foliated transport problem can be understood as a family of independent optimal transport problems, one on each leaf of $\mathcal F$. Indeed, the global foliated cost is the average, with respect to the common transverse measure $\sigma$, of the optimal costs between the corresponding conditional measures. Since no mass is exchanged between distinct leaves, the only connection among the leafwise problems is their measurable dependence on the transverse parameter $y$.
	
	There is also an equivalent way of encoding the leafwise
	constraint. Recall that the extended leafwise distance was
	defined by setting $d_{\mathcal L}(x,x')=+\infty$ whenever $x$ and $x'$ belong to different leaves. Consequently, instead of restricting the admissible class to foliated transport
	plans, one may minimize the extended cost $d_{\mathcal L}^{2}$ over the entire set $\Pi(\mu,\nu)$ of transport plans. Plans that move a positive amount of mass between different leaves automatically have infinite cost. The following theorem establishes that these two formulations are equivalent.
	
	\
	
	\begin{theorem} \label{theorem:extended-cost-formulation}
		Let $X$ be a Riemannian manifold endowed with a Borel
		foliation $\mathcal F$. Then, for every
		$\mu$, $\nu\in\mathscr P(X)$, we have $W_{2,\mathcal F}^{2}(\mu,\nu)	= \widetilde{W}_2^2 (\mu, \nu)$, where
		\begin{displaymath}
			 \widetilde{W}_2^2 (\mu, \nu) = \inf_{\gamma\in\Pi(\mu,\nu)} \int_{X\times X} d_{\mathcal L}^{2}(x,x')\,d\gamma(x,x')
		\end{displaymath}
		is the classical Kantorovich problem over $\Pi(\mu,\nu)$ with the extended-valued cost $d_{\mathcal L}^{2}$.
	\end{theorem}
	
	\begin{proof}
		Recall that $E_{\mathcal F} = \bigl\{(x,x')\in X\times X: \pi(x)=\pi(x') \bigr\}$, and note that $\Pi_{\mathcal F}(\mu,\nu)	= \bigl\{\gamma\in\Pi(\mu,\nu): \gamma(E_{\mathcal F})=1 \bigr\}$. Let $\gamma\in\Pi(\mu,\nu)$. If $\gamma\notin\Pi_{\mathcal F}(\mu,\nu)$, then $\gamma(E_{\mathcal F}^{c})>0$.	By the definition of the extended leafwise distance, $d_{\mathcal L}(x,x')=+\infty$ for every $(x,x')\in E_{\mathcal F}^{c}$. Therefore, $d_{\mathcal L}^{2}=+\infty$ on a set of positive $\gamma$-measure, and hence
		\begin{displaymath}
			\int_{X\times X}
			d_{\mathcal L}^{2}(x,x')\,d\gamma(x,x')
			=
			+\infty.
		\end{displaymath}
		On the other hand, if $\gamma\in\Pi_{\mathcal F}(\mu,\nu)$, then $\gamma$ is concentrated on $E_{\mathcal F}$, and its cost is precisely the cost appearing in the definition of $W_{2,\mathcal F}^{2}(\mu,\nu)$. Consequently, the only plans in
		$\Pi(\mu,\nu)$ that may have finite extended leafwise cost are
		the foliated transport plans. It follows that
		\begin{align*}
			\inf_{\gamma\in\Pi(\mu,\nu)} ~
			\int_{X\times X} d_{\mathcal L}^{2}(x,x')\,d\gamma(x,x')
			&= \inf_{\gamma\in\Pi_{\mathcal F}(\mu,\nu)} ~
			\int_{X\times X} d_{\mathcal L}^{2}(x,x')\,d\gamma(x,x')\\
			&= W_{2,\mathcal F}^{2}(\mu,\nu).
		\end{align*}
		If $\Pi_{\mathcal F}(\mu,\nu)=\varnothing$, then every
		$\gamma\in\Pi(\mu,\nu)$ satisfies $\gamma(E_{\mathcal F}^{c})>0$
		and has infinite extended leafwise cost. Therefore, in this
		case, both sides of the asserted equality are equal to $+\infty$.
	\end{proof}
	
	\
	
	Theorem \ref{theorem:extended-cost-formulation} places the foliated problem within the classical Kantorovich theory. As a first consequence, optimal foliated transport plans exist as soon as the extended leafwise distance is lower semicontinuous, and in particular whenever $X$ is complete, by Lemma \ref{lemma:dL-Borel-lsc}.
	
	\
	
	\begin{corollary} \label{corollary:existence-optimal-plans}
		Let $X$ be a Riemannian manifold endowed with a Borel foliation $\mathcal F$ of class $C^1$, and assume that $X$ is complete or, more generally, that $d_{\mathcal L}:X\times X\to[0,+\infty]$ is lower semicontinuous. Let $\mu,\nu\in\mathscr P(X)$ satisfy $W_{2,\mathcal F}^2(\mu,\nu)<+\infty$. Then there exists $\gamma^\star\in\Pi_{\mathcal F}(\mu,\nu)$ such that
		\begin{displaymath}
			\int_{X\times X}d_{\mathcal L}^2(x,x')\,d\gamma^\star(x,x')=W_{2,\mathcal F}^2(\mu,\nu).
		\end{displaymath}
	\end{corollary}
	
	\begin{proof}
		If $X$ is complete, then $d_{\mathcal L}$ is lower semicontinuous by Lemma \ref{lemma:dL-Borel-lsc}. The manifold $X$ is a Polish space, and the cost $d_{\mathcal L}^2:X\times X\to[0,+\infty]$ is lower semicontinuous and nonnegative. By \cite[Theorem 4.1]{Vil09}, applied with the lower bounds $a=b=0$, there exists $\gamma^\star\in\Pi(\mu,\nu)$ minimizing $\int d_{\mathcal L}^2\,d\gamma$ over $\Pi(\mu,\nu)$. By Theorem \ref{theorem:extended-cost-formulation},
		\begin{displaymath}
			\int_{X\times X}d_{\mathcal L}^2(x,x')\,d\gamma^\star(x,x')=\widetilde W_2^2(\mu,\nu)=W_{2,\mathcal F}^2(\mu,\nu)<+\infty.
		\end{displaymath}
		Since $d_{\mathcal L}=+\infty$ on $E_{\mathcal F}^c$, it follows that $\gamma^\star(E_{\mathcal F}^c)=0$, that is, $\gamma^\star\in\Pi_{\mathcal F}(\mu,\nu)$.
	\end{proof}
	
	\
	
	The preceding results provide complementary descriptions of	the foliated transport problem. Theorem \ref{Theorem-2} decomposes its cost into a measurable family of optimal transport problems on the individual leaves, whereas Theorem \ref{theorem:extended-cost-formulation} incorporates the leafwise constraint into an extended-valued cost and recovers the problem as an unconstrained Kantorovich problem. We now illustrate these descriptions in some concrete geometric settings. The first example concerns foliations induced by smooth submersions, for which the leaves are given by the fibers of the corresponding	map and the transverse measure is obtained by pushing the marginals forward to the base space.
	
	\
	
	\begin{example}	\label{example:submersion-foliation}
		Let $X$ be a Riemannian manifold and $B$ a smooth manifold, and let $\pi:X \to B$ be a smooth surjective submersion with connected fibers. The fibers of $\pi$ define a foliation $\mathcal F$ of $X$, whose leaves are $\mathcal L_y:=\pi^{-1}(\{y\})$, $y\in B$. Since every manifold is a standard Borel space and $\pi$ is continuous, $\mathcal F$ is a Borel foliation in the sense of Definition \ref{Def-Borel-foliation}. A simple instance of this construction is obtained by considering the flat torus $X=\mathbb T^2=\mathbb S^1\times\mathbb S^1$, where $\mathbb S^1=\mathbb R/\mathbb Z$, and the projection onto the second coordinate, $\pi:\mathbb T^2 \to \mathbb S^1$ such that $\pi(x,y)=y$.	The leaves of the corresponding foliation are the horizontal circles $\mathcal L_y =	\mathbb S^1\times\{y\}$, $y\in\mathbb S^1$. In this case, the extended leafwise distance is given by
		\begin{displaymath}
			d_{\mathcal L}\bigl((x,y),(x',y')\bigr)=
			\begin{cases}
				d_{\mathbb S^1}(x,x') & \text{if~~}~~~~ y=y',\\
				+\infty & \text{otherwise},
			\end{cases}
		\end{displaymath}
		where $d_{\mathbb S^1}$ denotes the intrinsic distance on $\mathbb S^1$. It is continuous on the closed set $\{y=y'\}$ and equal to $+\infty$ on its complement, hence lower semicontinuous, in accordance with Lemma \ref{lemma:dL-Borel-lsc}. Moreover, $d_{\mathcal L}\leq1/2$ on $E_{\mathcal F}$, so $W_{2,\mathcal F}^2(\mu,\nu)\leq1/4$ whenever $\pi_*\mu=\pi_*\nu$. By Corollary \ref{corollary:existence-optimal-plans}, an optimal foliated transport plan between $\mu$ and $\nu$ exists if and only if $\pi_*\mu=\pi_*\nu$. If $\mu,\nu\in\mathscr P(\mathbb T^2)$ satisfy $\pi_*\mu=\pi_*\nu=:\sigma$, then
		\begin{displaymath}
			\mu=\int_{\mathbb S^1}\mu^y\,d\sigma(y),
			\qquad
			\nu=\int_{\mathbb S^1}\nu^y\,d\sigma(y),
		\end{displaymath}
		and
		\begin{displaymath}
			W_{2,\mathcal F}^{2}(\mu,\nu)
			=
			\int_{\mathbb S^1}
			W_{2,\mathcal L_y}^{2}(\mu^y,\nu^y)\,d\sigma(y).
		\end{displaymath}
		Here, each conditional transport problem takes place on the
		horizontal circle $\mathcal L_y$, and its cost is computed using
		the intrinsic distance on that circle.
	\end{example}
	
	\
	
	The preceding example concerns the regular situation in which the leaves are the fibers of a smooth submersion onto a smooth manifold. The Borel formulation considered here, however, is not restricted to this setting: the transverse space $X^*$ is only required to be a standard Borel space and need not carry a smooth manifold structure. This allows us to include foliations whose leaf spaces are singular from the differentiable viewpoint, provided that they remain regular from the measurable viewpoint, in the sense of Definition \ref{Def-Borel-foliation} (see Remark \ref{remark:scope-Borel-foliation}). The following example illustrates this broader situation through the orbit foliation of a locally free circle action, for which the leaf space is naturally an orbifold rather than a smooth manifold.
	
	\
	
	\begin{example}	\label{example:S3-foliation}
		Let $X=\mathbb S^3 = \bigl\{ (z_1,z_2)\in\mathbb C^2: |z_1|^2+|z_2|^2=1 \bigr\}$	endowed with its standard Riemannian metric. Fix relatively prime integers $p,q\geq 2$, and consider the smooth action of the circle group $\mathbb S^1=\{e^{it}:t\in\mathbb R\}\subset\mathbb C$ on $\mathbb S^3$ defined by $e^{it}\cdot(z_1,z_2) = \bigl(e^{ipt}z_1,e^{iqt}z_2\bigr)$. Since the generating vector field $V(z_1,z_2) = \bigl(ipz_1,iqz_2\bigr)$	does not vanish on $\mathbb S^3$, the orbits of this action define a one-dimensional foliation $\mathcal F$. The leaf through $x\in\mathbb S^3$ is 
		\begin{displaymath}
			\mathcal L_x = \bigl\{e^{it}\cdot x:t\in\mathbb R	\bigr\}.
		\end{displaymath}
		Let $X^*:=\mathbb S^3/\mathbb S^1$	be the orbit space, and let $\pi:\mathbb S^3\to X^*$ such that $\pi(x)=\mathcal L_x$ be the quotient map. Since $\mathbb S^1$ is compact and acts	continuously on $\mathbb S^3$, the orbit space $X^*$ is compact and metrizable. In particular, it is a standard Borel space, and $\pi$ is a Borel map. Therefore, $\mathcal F$ is a Borel foliation. Notice that this foliation is not induced by a submersion onto a smooth manifold. Indeed, the orbits through $(1,0)$ and $(0,1)$ are exceptional, with isotropy groups $\mathbb Z_p$ and $\mathbb Z_q$, respectively, and this is reflected in their holonomy \cite{CC00}. Consider the transverse disc $D:=\bigl\{\bigl(\sqrt{1-|w|^2},w\bigr):w\in\mathbb C,\ |w|<\delta\bigr\}$ to the orbit through $(1,0)$, for small $\delta>0$. The orbit of a point $\bigl(\sqrt{1-|w|^2},w\bigr)\in D$ returns to $D$ for the first time at $t=2\pi/p$, at the point $\bigl(\sqrt{1-|w|^2},e^{2\pi iq/p}w\bigr)$. Hence the holonomy of this leaf is generated by the rotation $w\mapsto e^{2\pi iq/p}w$, which has order $p$ because $p$ and $q$ are relatively prime. Similarly, the orbit through $(0,1)$ has holonomy of order $q$. On the other hand, a submersion defined on the compact manifold $\mathbb S^3$ is proper, so by Ehresmann's theorem \cite{Ehr51} its fibers form a locally trivial fibration, and every leaf of the corresponding foliation has trivial holonomy. The orbit space carries instead the structure of a weighted projective line, or equivalently of a two-dimensional orbifold with two cone points, of orders $p$ and $q$.
		
		\noindent Since $\mathbb S^3$ is compact, $d_{\mathcal L}$ is lower semicontinuous by Lemma \ref{lemma:dL-Borel-lsc}. In this example it can also be computed explicitly. For $x=(z_1,z_2)\in\mathbb S^3$, the moduli $|z_1|$ and $|z_2|$ are invariant under the action, so the curve $t\mapsto e^{it}\cdot x$ has constant speed
		\begin{displaymath}
			c(x):=|V(x)|=\bigl(p^2|z_1|^2+q^2|z_2|^2\bigr)^{1/2}>0.
		\end{displaymath}
		This curve is periodic, with minimal period $T(x)$ equal to $2\pi/p$ on the orbit through $(1,0)$, to $2\pi/q$ on the orbit through $(0,1)$, and to $2\pi$ on every other orbit. Hence $\mathcal L_x$ is a circle of length $c(x)T(x)$, and, if $x'=e^{it_0}\cdot x$, then $d_{\mathcal L}(x,x')=c(x)\operatorname{dist}\bigl(t_0,T(x)\mathbb Z\bigr)$. Since $T(x)$ divides $2\pi$, this distance is attained at some $t\in[-\pi,\pi]$, and therefore
		\begin{displaymath}
			d_{\mathcal L}(x,x')=\min\bigl\{c(x)|t|:t\in[-\pi,\pi],\ e^{it}\cdot x=x'\bigr\},
		\end{displaymath}
		with the convention $\min\varnothing=+\infty$, which covers the case where $x$ and $x'$ lie on different orbits. The set 
		\begin{displaymath}
			\{(x,x',t)\in\mathbb S^3\times\mathbb S^3\times[-\pi,\pi]:e^{it}\cdot x=x'\}
		\end{displaymath}
		is compact and the function $(x,t)\mapsto c(x)|t|$ is continuous, which confirms directly that $d_{\mathcal L}$ is lower semicontinuous on $\mathbb S^3\times\mathbb S^3$.
		
		\noindent Moreover, $c(x)\leq\max\{p,q\}$, so $d_{\mathcal L}\leq\pi\max\{p,q\}$ on $E_{\mathcal F}$, and $W_{2,\mathcal F}^2(\mu,\nu)$ is finite whenever $\pi_*\mu=\pi_*\nu$. By Corollary \ref{corollary:existence-optimal-plans}, an optimal foliated transport plan between $\mu$ and $\nu$ exists if and only if $\pi_*\mu=\pi_*\nu$. Let $\mu,\nu\in\mathscr P(\mathbb S^3)$ satisfy $\pi_*\mu=\pi_*\nu=:\sigma$, and let
		\begin{displaymath}
			\mu=\int_{X^*}\mu^y\,d\sigma(y),
			\qquad
			\nu=\int_{X^*}\nu^y\,d\sigma(y)
		\end{displaymath}
		be their disintegrations with respect to $\pi$. For
		$\sigma$-almost every $y\in X^*$, the measures $\mu^y$ and
		$\nu^y$ are concentrated on the orbit $\mathcal L_y:=\pi^{-1}(\{y\})$. Therefore, Theorem \ref{Theorem-2} yields
		\begin{displaymath}
			W_{2,\mathcal F}^{2}(\mu,\nu)
			=
			\int_{X^*}
			W_{2,\mathcal L_y}^{2}(\mu^y,\nu^y)\,d\sigma(y).
		\end{displaymath}
		Each term on the right-hand side is a one-dimensional optimal
		transport problem on a circular orbit, with cost computed using
		the intrinsic Riemannian distance along that orbit. Thus, the
		global transport problem on $\mathbb S^3$ decomposes into a
		measurable family of transport problems along the orbits of the
		weighted circle action.	Equivalently, by Theorem \ref{theorem:extended-cost-formulation},
		\begin{displaymath}
			W_{2,\mathcal F}^{2}(\mu,\nu)
			=
			\inf_{\gamma\in\Pi(\mu,\nu)}
			\int_{\mathbb S^3\times\mathbb S^3}
			d_{\mathcal L}^{2}(x,x')\,d\gamma(x,x'),
		\end{displaymath}
		where transport between points belonging to distinct
		$\mathbb S^1$-orbits is assigned infinite cost.
	\end{example}

	\
	
\section{Benamou-Brenier formulation} \label{Section:Benamou}
	
		The Benamou-Brenier formulation provides a dynamical interpretation of the quadratic optimal transport problem. Instead of describing transport by a coupling between the initial and final measures, it considers a time-dependent curve of probability measures governed by a continuity equation and minimizes the corresponding kinetic energy. Let $\mu_0,\mu_1\in\mathscr P_2(\mathbb R^d)$. Consider a narrowly continuous curve $[0,1]\ni t\mapsto \mu_t \in \mathscr P_2(\mathbb R^d)$, that is, for every bounded continuous function $\varphi:\mathbb R^d\to\mathbb R$, the map
		\begin{displaymath}
				[0,1]\ni t
				\mapsto
				\int_{\mathbb R^d}\varphi(x)\,d\mu_t(x)
		\end{displaymath}
		is continuous. Let $v:(0,1)\times\mathbb R^d\to\mathbb R^d$, $(t,x)\mapsto v_t(x)$ be a Borel vector field satisfying
		\begin{displaymath}
			\int_0^1\int_{\mathbb R^d} \lVert v_t(x)\rVert^2\,d\mu_t(x)\,dt<+\infty.
		\end{displaymath}
		We say that the pair $(\mu_t,v_t)$ satisfies the continuity equation if
		\begin{equation}\label{eq:classical-continuity-equation}
			\partial_t\mu_t+\operatorname{div}(\mu_tv_t)=0
		\end{equation}
		in the sense of distributions on $(0,1)\times\mathbb R^d$. Equivalently,
		\begin{displaymath}
			\int_0^1\int_{\mathbb R^d} \left( \partial_t\varphi(t,x) 	+ \langle\nabla\varphi(t,x),v_t(x)\rangle \right) d\mu_t(x) ~dt	= 0
		\end{displaymath}
		for every $\varphi\in C_c^\infty((0,1)\times\mathbb R^d)$. Denote by $\operatorname{CE}(\mu_0,\mu_1)$ the collection of all such pairs satisfying the endpoint conditions $\mu_{t=0}=\mu_0$ and $\mu_{t=1}=\mu_1$. The kinetic action of $(\mu_t,v_t)$ is defined by
		\begin{displaymath}
			\mathcal A(\mu,v) := \int_0^1\int_{\mathbb R^d}	\lVert v_t(x)\rVert^2\,d\mu_t(x)\,dt.
		\end{displaymath}
		The Benamou-Brenier formula \cite{BB00} states that
		\begin{equation} \label{eq:classical-benamou-brenier}
			W_2^2(\mu_0,\mu_1) = \inf \left\{ \mathcal A(\mu,v): (\mu_t,v_t)\in \operatorname{CE}(\mu_0,\mu_1) \right\}.
		\end{equation}
		Thus, the quadratic Wasserstein cost is the minimum kinetic energy required to evolve $\mu_0$ into $\mu_1$ through solutions of the continuity equation.
	
		The same formulation extends naturally to a Riemannian manifold $X$. In this case, the Euclidean norm and gradient are replaced by their Riemannian counterparts, and the divergence is understood with respect to the Riemannian structure. 
	
		When $X$ is additionally endowed with a Borel foliation $\mathcal F$, the dynamical formulation must reflect the geometric restriction that mass can move only along the leaves. This is achieved by	requiring the velocity field to satisfy $v_t(x)\in T_x\mathcal L_x$ for $\mu_t\,dt$-almost every $(t,x)$, where $T_x\mathcal L_x$ denotes the tangent space to the leaf through $x$. We now make this foliated version of the continuity equation precise.
		
		For every $x\in X$, let $P_{\mathcal F}(x):T_xX\longrightarrow T_x\mathcal L_x$ denote the orthogonal projection onto the tangent space of the
		leaf through $x$. Given a smooth function $\varphi:X\to\mathbb R$, we define its leafwise gradient by $\nabla_{\mathcal F}\varphi(x) := P_{\mathcal F}(x)\nabla\varphi(x)$. Thus, $\nabla_{\mathcal F}\varphi$ records only the variation of
		$\varphi$ along the leaves of $\mathcal F$.
		
		\
		
		\begin{definition} \label{def:foliated-continuity-equation} Let $[0,1]\ni t\mapsto\mu_t\in\mathscr P_2(X)$ be a narrowly continuous curve, and let $v:(0,1)\times X\to TX$ be a Borel time-dependent vector field satisfying $v_t(x)\in T_x\mathcal L_x$ for $\mu_t\,dt$-almost every $(t,x)\in(0,1)\times X$. We say that $(\mu_t,v_t)$ satisfies the foliated continuity equation if
			\begin{equation} \label{eq:foliated-continuity-equation}
				\partial_t\mu_t+
				\operatorname{div}_{\mathcal F}(\mu_tv_t)=0
			\end{equation}
			in the weak sense, that is, if
			\begin{displaymath}
				\int_0^1\int_X
				\left(
				\partial_t\varphi(t,x)
				+
				\left\langle
				\nabla_{\mathcal F}\varphi(t,x),v_t(x)
				\right\rangle
				\right)
				d\mu_t(x)\,dt
				=
				0
			\end{displaymath}
			for every
			$\varphi\in C_c^\infty((0,1)\times X)$.
		\end{definition}
		
		\		
		
		The leafwise divergence of the vector-valued measure $\mu_tv_t$
		is understood in the distributional sense and it is defined by
		\begin{displaymath}
			\left\langle
			\operatorname{div}_{\mathcal F}(\mu_tv_t),\psi
			\right\rangle
			:=
			-
			\int_X
			\left\langle
			\nabla_{\mathcal F}\psi(x),v_t(x)
			\right\rangle
			d\mu_t(x)
		\end{displaymath}
		for every $\psi\in C_c^\infty(X)$. Since $v_t(x)\in T_x\mathcal L_x$, we have
		\begin{displaymath}
			\left\langle\nabla\varphi(t,x),v_t(x)\right\rangle
			=
			\left\langle
			\nabla_{\mathcal F}\varphi(t,x),v_t(x)
			\right\rangle.
		\end{displaymath}
		Therefore, every solution of the foliated continuity equation is
		also a solution of the usual continuity equation on $X$. The
		distinction lies in the tangential constraint on the velocity
		field, which prevents mass from moving in directions transverse
		to the foliation.
		
		\
		
		The arguments below rely on a representation of solutions of the foliated continuity equation by probability measures on leafwise trajectories. The superposition principle \cite[Theorem 8.2.1]{AGS08} is formulated in $\mathbb R^d$, and we transfer it to $X$ by means of a closed isometric embedding, which exists whenever $X$ is complete \cite{Mul09}. Recall that $\mathcal F$ is of class $C^1$ and that $X$ is covered by $C^1$ foliated charts $\phi:U\to D^k\times D^\ell$ as in Section \ref{Section:Kantorovich}. In particular, the tangent distribution $x\mapsto T_x\mathcal L_x$ is continuous, and the set $\{(x,w)\in TX: w\in T_x\mathcal L_x\}$ is closed in $TX$.
		
		\
		
		\begin{lemma}[Leafwise superposition] \label{lemma:leafwise-superposition}
			Let $X$ be a complete Riemannian manifold, with Riemannian metric $g$, endowed with a Borel foliation $\mathcal F$ of class $C^1$, and let $(\mu_t,v_t)$ satisfy the foliated continuity equation with
			\begin{displaymath}
				\int_0^1\int_X\|v_t(x)\|_g^2\,d\mu_t(x)\,dt<+\infty.
			\end{displaymath}
			Then, there exists $\boldsymbol\eta\in\mathscr P\bigl(C([0,1];X)\bigr)$ such that
			\begin{enumerate}[leftmargin=*]
				\item[(i)] $(e_t)_*\boldsymbol\eta=\mu_t$ for every $t\in[0,1]$, where $e_t(\omega):=\omega(t)$.
				\item[(ii)] $\boldsymbol\eta$ is concentrated on absolutely continuous curves $\omega$ such that, for almost every $t\in(0,1)$,  $\dot\omega(t)=v_t(\omega(t))$ .
				\item[(iii)] For $\boldsymbol\eta$-almost every $\omega$, one has $\dot\omega(t)\in T_{\omega(t)}\mathcal L_{\omega(t)}$ for almost every $t$, and $\omega([0,1])\subset\mathcal L_{\omega(0)}$. In particular, $\pi(\omega(t))=\pi(\omega(0))$ for every $t\in[0,1]$.
			\end{enumerate}
		\end{lemma}
		
		\begin{proof}
			By the observation following Definition \ref{def:foliated-continuity-equation}, $(\mu_t,v_t)$ solves the usual continuity equation on $X$, that is,
			\begin{equation}\label{eq:CE-on-X}
				\int_0^1\int_X\bigl(\partial_t\varphi(t,x)+\langle\nabla\varphi(t,x),v_t(x)\rangle_g\bigr)\,d\mu_t(x)\,dt=0
			\end{equation}
			for every $\varphi\in C_c^\infty((0,1)\times X)$.
			
			\
			
			\noindent \emph{Step 1: transfer to a Euclidean space.} Since $X$ is complete, by \cite{Mul09} there exists a smooth isometric embedding $\iota:X\to\mathbb R^N$ whose image $\iota(X)$ is closed. In particular, $\iota$ is proper: if $K\subset\mathbb R^N$ is compact, then $K\cap\iota(X)$ is compact, and so is $\iota^{-1}(K)=\iota^{-1}(K\cap\iota(X))$, because $\iota$ is a homeomorphism onto its image. Set $\tilde\mu_t:=\iota_*\mu_t$ and define the Borel vector field
			\begin{displaymath}
				\begin{cases}
					\tilde v_t(z):=d\iota_{\iota^{-1}(z)}\bigl(v_t(\iota^{-1}(z))\bigr) & \text{if } z\in\iota(X) \\
					\tilde v_t(z):=0 &\text{otherwise}.
				\end{cases}			
			\end{displaymath}
			The curve $t\mapsto\tilde\mu_t$ is narrowly continuous, since $f\circ\iota$ is bounded and continuous on $X$ whenever $f$ is bounded and continuous on $\mathbb R^N$. Since $\iota$ is isometric, $|\tilde v_t(\iota(x))|=\|v_t(x)\|_g$, and therefore
			\begin{displaymath}
				\int_0^1\int_{\mathbb R^N}|\tilde v_t(z)|^2\,d\tilde\mu_t(z)\,dt=\int_0^1\int_X\|v_t(x)\|_g^2\,d\mu_t(x)\,dt<+\infty.
			\end{displaymath}
			Let $\psi\in C_c^\infty((0,1)\times\mathbb R^N)$, with support contained in $[a,b]\times K$ for some $0<a<b<1$ and some compact set $K\subset\mathbb R^N$, and set $\varphi(t,x):=\psi(t,\iota(x))$. Then $\varphi$ is smooth and its support is contained in $[a,b]\times\iota^{-1}(K)$, which is compact by properness of $\iota$. Thus $\varphi\in C_c^\infty((0,1)\times X)$. Moreover,
			\begin{displaymath}
				\partial_t\varphi(t,x)=\partial_t\psi(t,\iota(x))
			\end{displaymath}
			and
			\begin{displaymath}
				\langle\nabla\varphi(t,x),w\rangle_g=\langle\nabla\psi(t,\iota(x)),d\iota_x(w)\rangle
			\end{displaymath}
			for every $w\in T_xX$. Applying \eqref{eq:CE-on-X} to $\varphi$, we obtain
			\begin{displaymath}
				\int_0^1\int_{\mathbb R^N}\bigl(\partial_t\psi(t,z)+\langle\nabla\psi(t,z),\tilde v_t(z)\rangle\bigr)\,d\tilde\mu_t(z)\,dt=0,
			\end{displaymath}
			so $(\tilde\mu_t,\tilde v_t)$ solves the continuity equation in $\mathbb R^N$.
			
			\
			
			\noindent \emph{Step 2:} By \cite[Theorem 8.2.1]{AGS08}, after projecting the measure given there onto the space of curves, there exists $\tilde{\boldsymbol\eta}\in\mathscr P\bigl(C([0,1];\mathbb R^N)\bigr)$ such that $(e_t)_*\tilde{\boldsymbol\eta}=\tilde\mu_t$ for every $t\in[0,1]$, and $\tilde{\boldsymbol\eta}$ is concentrated on absolutely continuous curves $\tilde\omega$ satisfying $\dot{\tilde\omega}(t)=\tilde v_t(\tilde\omega(t))$ for almost every $t\in(0,1)$.
			
			\
			
			\noindent \emph{Step 3: the curves lie in $\iota(X)$ and lift to $X$.} For every $t\in[0,1]$,
			\begin{displaymath}
				\tilde{\boldsymbol\eta}\bigl(\{\tilde\omega:\tilde\omega(t)\notin\iota(X)\}\bigr)=\tilde\mu_t\bigl(\mathbb R^N\setminus\iota(X)\bigr)=0.
			\end{displaymath}
			Taking the union over $t\in\mathbb Q\cap[0,1]$, we see that, for $\tilde{\boldsymbol\eta}$-almost every $\tilde\omega$, one has $\tilde\omega(t)\in\iota(X)$ for every rational $t$, and hence for every $t\in[0,1]$, by continuity of $\tilde\omega$ and closedness of $\iota(X)$. Thus $\tilde{\boldsymbol\eta}$ is concentrated on the closed set $C([0,1];\iota(X))$, on which the map $\Phi(\tilde\omega):=\iota^{-1}\circ\tilde\omega\in C([0,1];X)$ is continuous. Define $\boldsymbol\eta:=\Phi_*\tilde{\boldsymbol\eta}$. Then $(e_t)_*\boldsymbol\eta=(\iota^{-1})_*\tilde\mu_t=\mu_t$ for every $t\in[0,1]$, which proves (i).
			
			\
			
			Let $V\subset\mathbb R^N$ be a tubular neighborhood of the embedded submanifold $\iota(X)$, and let $r:V\to\iota(X)$ be the associated smooth retraction. The map $F:=\iota^{-1}\circ r:V\to X$ is smooth and satisfies $F\circ\iota=\operatorname{id}_X$, so that $dF_{\iota(x)}\circ d\iota_x=\operatorname{id}_{T_xX}$ for every $x\in X$. Let $\tilde\omega$ be an absolutely continuous curve with values in $\iota(X)$ satisfying $\dot{\tilde\omega}(t)=\tilde v_t(\tilde\omega(t))$ for almost every $t$, and let $\omega:=\Phi(\tilde\omega)=F\circ\tilde\omega$. Since $\tilde\omega([0,1])$ is compact and $F$ is smooth on a neighborhood of it, $\omega$ is absolutely continuous and, by the chain rule, $\dot\omega(t)=dF\bigl(\dot{\tilde\omega}(t)\bigr)$ for almost every $t$. For almost every $t$ we also have $\dot{\tilde\omega}(t)=\tilde v_t(\tilde\omega(t))=d\iota\bigl(v_t(\omega(t))\bigr)$, and consequently $\dot\omega(t)=dF\circ d\iota\bigl(v_t(\omega(t))\bigr)=v_t(\omega(t))$. This proves (ii).
			
			\
			
			\noindent \emph{Step 4: tangency along the curves.} Since $v$ is Borel and the tangent distribution of $\mathcal F$ is closed in $TX$, the set
			\begin{displaymath}
				B:=\bigl\{(t,x)\in(0,1)\times X: v_t(x)\notin T_x\mathcal L_x\bigr\}
			\end{displaymath}
			is Borel. Set $B_t:=\{x\in X:(t,x)\in B\}$. By Definition \ref{def:foliated-continuity-equation}, we have $\int_0^1\mu_t(B_t)\,dt=0$, and hence, by Fubini's theorem and (i),
			\begin{displaymath}
				\int_{C([0,1];X)}\int_0^1\mathbf 1_B(t,\omega(t))\,dt\,d\boldsymbol\eta(\omega)=\int_0^1\mu_t(B_t)\,dt=0.
			\end{displaymath}
			Hence, for $\boldsymbol\eta$-almost every $\omega$, we have $v_t(\omega(t))\in T_{\omega(t)}\mathcal L_{\omega(t)}$ for almost every $t$ and therefore, by (ii), $\dot\omega(t)\in T_{\omega(t)}\mathcal L_{\omega(t)}$ for almost every $t$.
			
			\
			
			\noindent \emph{Step 5: such curves remain in a single leaf.} Let $\omega:[0,1]\to X$ be absolutely continuous with $\dot\omega(t)\in T_{\omega(t)}\mathcal L_{\omega(t)}$ for almost every $t$. Cover the compact set $\omega([0,1])$ by finitely many $C^1$ foliated charts $\phi_i:U_i\to D^k\times D^\ell$. By a Lebesgue number argument, there is a partition $0=t_0<t_1<\dots<t_J=1$ such that, for each $j$, $\omega([t_{j-1},t_j])\subset U_{i_j}$ for some index $i_j$. On $[t_{j-1},t_j]$, the transverse coordinate $z(t):=\operatorname{pr}_2\circ\phi_{i_j}\circ\omega(t)$ is absolutely continuous and, since $\ker d(\operatorname{pr}_2\circ\phi_{i_j})_x=T_x\mathcal L_x$ for every $x\in U_{i_j}$,
			\begin{displaymath}
				\dot z(t)=d(\operatorname{pr}_2\circ\phi_{i_j})\bigl(\dot\omega(t)\bigr)=0
			\end{displaymath}
			for almost every $t$ in $[t_{j-1},t_j]$. Hence $z$ is constant on this interval, and $\omega([t_{j-1},t_j])$ is contained in a single plaque of $U_{i_j}$, and thus in a single leaf. Since consecutive intervals share an endpoint, $\omega([0,1])\subset\mathcal L_{\omega(0)}$, and therefore $\pi(\omega(t))=\pi(\omega(0))$ for every $t\in[0,1]$. Together with Step 4, this proves (iii).
		\end{proof}
		
		\
		
		We emphasize that Lemma \ref{lemma:leafwise-superposition} represents the curve $(\mu_t)$ through trajectories driven by the given velocity field $v_t$, and not merely through some lift whose speed is controlled. This is essential in what follows: a lift of $(\mu_t)$ controlled only by its metric speed follows the velocity field of minimal norm, which need not be tangent to the leaves, and whose trajectories may therefore pass from one leaf to another even though $\pi_*\mu_t$ is conserved.
		
		The tangential constraint on the velocity field has an immediate	consequence for the evolution of the transverse distribution of mass. Since the velocity has no component normal to the leaves, transport trajectories cannot move from one leaf to another. Therefore, although the conditional distribution of mass within each leaf may evolve in time, the total amount of mass assigned to each leaf must remain unchanged. In terms of the map $\pi:X\to X^*$ of Definition \ref{Def-Borel-foliation}, this means that the transverse marginal $\pi_*\mu_t$ is constant along every solution of the foliated continuity equation. We make this conservation property precise in the following proposition.

		\
	
	\begin{proposition}	\label{proposition:transverse-conservation}
		Let $X$ be a complete Riemannian manifold endowed with a Borel foliation $\mathcal F$ of class $C^1$. Let $(\mu_t,v_t)$ satisfy the foliated continuity equation and
		assume that
		\begin{displaymath}
			\int_0^1\int_X
			\|v_t(x)\|_g^2\,d\mu_t(x)\,dt
			<
			+\infty.
		\end{displaymath}
		Then $\pi_*\mu_t=\pi_*\mu_0$ for every $t\in[0,1]$.
	\end{proposition}
	
	\begin{proof}
		By Lemma \ref{lemma:leafwise-superposition}, there exists a probability measure $\boldsymbol{\eta}$ on $C([0,1];X)$ such that $(e_t)_*\boldsymbol{\eta}=\mu_t$ for every $t\in[0,1]$, where $e_t(\omega):=\omega(t)$, and such that $\pi(\omega(t))=\pi(\omega(0))$ for every $t\in[0,1]$ and for $\boldsymbol{\eta}$-almost every $\omega$.
		
		Let $f:X^*\to\mathbb R$ be bounded and Borel measurable. Then
		\begin{align*}
			\int_{X^*}f(y)\,d(\pi_*\mu_t)(y)
			&=
			\int_Xf(\pi(x))\,d\mu_t(x)\\
			&=
			\int_{C([0,1];X)}
			f(\pi(\omega(t)))\,d\boldsymbol{\eta}(\omega)\\
			&=
			\int_{C([0,1];X)}
			f(\pi(\omega(0)))\,d\boldsymbol{\eta}(\omega)\\
			&=
			\int_Xf(\pi(x))\,d\mu_0(x)\\
			&=
			\int_{X^*}f(y)\,d(\pi_*\mu_0)(y).
		\end{align*}
		Therefore, $\pi_*\mu_t=\pi_*\mu_0$ for every $t\in[0,1]$.
	\end{proof}
	
	\
	
	For a solution $(\mu_t,v_t)$ of the foliated continuity equation,
	we define its kinetic action by
	\begin{displaymath}
		\mathcal A_{\mathcal F}(\mu,v) := \int_0^1\int_X \|v_t(x)\|_g^2\,d\mu_t(x) ~dt.
	\end{displaymath}
	Given $\mu_0,\mu_1\in\mathscr P_2(X)$, denote by
	$\operatorname{CE}_{\mathcal F}(\mu_0,\mu_1)$ the set of all pairs $(\mu_t,v_t)$ satisfying the foliated continuity equation, having finite kinetic action and satisfying the endpoint conditions $\mu_{t=0}=\mu_0$, $\mu_{t=1}=\mu_1$. The foliated dynamical transport cost is defined by
	\begin{displaymath}
		\mathcal W_{2,\mathcal F}^{2}(\mu_0,\mu_1)
		:=
		\inf
		\left\{
		\mathcal A_{\mathcal F}(\mu,v):
		(\mu_t,v_t)\in
		\operatorname{CE}_{\mathcal F}(\mu_0,\mu_1)
		\right\}.
	\end{displaymath}
	As usual, the infimum over the empty set is understood to be
	$+\infty$. By Proposition \ref{proposition:transverse-conservation}, a
	necessary condition for $\operatorname{CE}_{\mathcal F}(\mu_0,\mu_1)$ to be nonempty is $\pi_*\mu_0=\pi_*\mu_1$. This is precisely the compatibility condition obtained for the existence of a foliated transport plan.

	In the converse direction to Lemma \ref{lemma:leafwise-superposition}, a probability measure on leafwise curves with finite energy induces an admissible solution of the foliated continuity equation, whose velocity field is obtained by averaging the velocities of the curves passing through each point.
	
	\
	
	\begin{lemma} \label{lemma:eulerian-representation}
		Let $X$ be a complete Riemannian manifold endowed with a Borel foliation $\mathcal F$ of class $C^1$. Let $\boldsymbol\eta$ be a probability measure on $C([0,1];X)$ concentrated on absolutely continuous curves, each of which is contained in a single leaf, and assume that
		\begin{displaymath}
			\mathcal E(\boldsymbol\eta):=\int_{C([0,1];X)}\int_0^1\|\dot\omega(t)\|_g^2\,dt\,d\boldsymbol\eta(\omega)<+\infty.
		\end{displaymath}
		Set $\mu_t:=(e_t)_*\boldsymbol\eta$ for $t\in[0,1]$, and assume that $\mu_0\in\mathscr P_2(X)$. Then the following hold.
		\begin{enumerate}[leftmargin=*]
			\item[(i)] The curve $t\mapsto\mu_t$ is narrowly continuous, and $\mu_t\in\mathscr P_2(X)$ for every $t\in[0,1]$.
			\item[(ii)] There exists a Borel vector field $v:(0,1)\times X\to TX$ such that $(\mu_t,v_t)$ satisfies the foliated continuity equation.
			\item[(iii)] $\displaystyle\int_0^1\int_X\|v_t(x)\|_g^2\,d\mu_t(x)\,dt\leq\mathcal E(\boldsymbol\eta)$.
		\end{enumerate}
	\end{lemma}

	\
	
	\begin{proof}
		
		\
		
		\
		
		\noindent \emph{Proof of (i).} If $f:X\to\mathbb R$ is bounded and continuous, then 
		$$t\mapsto\int_Xf\,d\mu_t=\int f(\omega(t))\,d\boldsymbol\eta(\omega)$$
		is continuous by dominated convergence, since every $\omega$ is continuous. Fix $x_0\in X$. For every absolutely continuous curve $\omega$ and every $t\in[0,1]$, the Cauchy-Schwarz inequality gives
		\begin{displaymath}
			d(\omega(t),x_0)\leq d(\omega(0),x_0)+\int_0^1\|\dot\omega\|_g\,ds\leq d(\omega(0),x_0)+\Bigl(\int_0^1\|\dot\omega\|_g^2\,ds\Bigr)^{1/2}.
		\end{displaymath}
		Squaring and integrating with respect to $\boldsymbol\eta$, we obtain 
		$$\int_Xd^2(x,x_0)\,d\mu_t(x)\leq2\int_Xd^2(x,x_0)\,d\mu_0(x)+2\,\mathcal E(\boldsymbol\eta)<+\infty.$$
		
		\
		
		\noindent \emph{Construction of $v$.} Let $\iota:X\to\mathbb R^N$ be a closed isometric embedding, as in the proof of Lemma \ref{lemma:leafwise-superposition}, and let $P_x:\mathbb R^N\to d\iota_x(T_xX)$ denote the orthogonal projection, which depends continuously on $x$. Let $D$ be the set of pairs $(t,\omega)\in(0,1)\times C([0,1];X)$ such that $\iota\circ\omega$ is differentiable at $t$. Writing derivatives as limits of difference quotients along rational increments, one checks that $D$ is Borel and that
		\begin{displaymath}
			\begin{cases}
					w(t,\omega):=(\iota\circ\omega)'(t) &\text{ if }(t,\omega)\in D \\
					w(t,\omega):=0 & \text{ otherwise}
			\end{cases}
		\end{displaymath}
		is a Borel map with values in $\mathbb R^N$. Consider the probability measure 
		$$\lambda:=\mathscr L^1|_{(0,1)}\otimes\boldsymbol\eta$$
		on $(0,1)\times C([0,1];X)$ and the continuous map $\Psi(t,\omega):=(t,\omega(t))$. Then $\Psi_*\lambda=\mu_t\,dt$, that is, 
		$$\int\varphi\,d(\Psi_*\lambda)=\int_0^1\int_X\varphi(t,x)\,d\mu_t(x)\,dt$$
		for every bounded Borel $\varphi$. For $\boldsymbol\eta$-almost every $\omega$ and almost every $t$, we have $(t,\omega)\in D$ and $w(t,\omega)=d\iota_{\omega(t)}(\dot\omega(t))$, so that $|w(t,\omega)|=\|\dot\omega(t)\|_g$ because $\iota$ is isometric. Moreover, since $\omega$ is an absolutely continuous curve contained in a single leaf, $\dot\omega(t)$ is its velocity as a curve in that leaf (see the discussion following the definition of $d_{\mathcal L}$ in Section \ref{Section:Kantorovich}), and hence $\dot\omega(t)\in T_{\omega(t)}\mathcal L_{\omega(t)}$. Consequently, for $\lambda$-almost every $(t,\omega)$,
		\begin{equation}\label{eq:w-tangent}
			w(t,\omega)\in V_{\omega(t)}:=d\iota_{\omega(t)}\bigl(T_{\omega(t)}\mathcal L_{\omega(t)}\bigr),
			\qquad
			\int|w|^2\,d\lambda=\mathcal E(\boldsymbol\eta)<+\infty.
		\end{equation}
		
		Since $(0,1)\times C([0,1];X)$ and $(0,1)\times X$ are Polish spaces, Theorem \ref{theorem:disintegration} gives a disintegration $\{\lambda_{t,x}\}$ of $\lambda$ with respect to $\Psi$: the measure $\lambda_{t,x}$ is concentrated on $\Psi^{-1}(\{(t,x)\})$ for $\mu_t\,dt$-almost every $(t,x)$, and
		$$\int F\,d\lambda=\int_0^1\int_X\int F\,d\lambda_{t,x}\,d\mu_t(x)\,dt$$
		for every nonnegative Borel $F$. By \eqref{eq:w-tangent}, for $\mu_t\,dt$-almost every $(t,x)$ we have $\int|w|^2\,d\lambda_{t,x}<+\infty$ and $w\in V_x$ for $\lambda_{t,x}$-almost every point. Define
		\begin{displaymath}
			\begin{cases}
				\tilde v(t,x):=\int w\,d\lambda_{t,x} & \text{ if }\int|w|\,d\lambda_{t,x}<+\infty \\
				\tilde v(t,x):=0 & \text{ otherwise},
			\end{cases}
		\end{displaymath}
		which is a Borel map, and $v_t(x):=(d\iota_x)^{-1}\bigl(P_x\tilde v(t,x)\bigr)\in T_xX$.
		
		\
		
		\noindent \emph{Proof of (iii).} For $\mu_t\,dt$-almost every $(t,x)$, $\tilde v(t,x)$ is the barycenter of a probability measure concentrated on the linear subspace $V_x\subset d\iota_x(T_xX)$, hence $\tilde v(t,x)\in V_x$. Therefore $P_x\tilde v(t,x)=\tilde v(t,x)$ and $v_t(x)\in T_x\mathcal L_x$ for $\mu_t\,dt$-almost every $(t,x)$. By Jensen's inequality, 
		$$\|v_t(x)\|_g^2=|\tilde v(t,x)|^2\leq\int|w|^2\,d\lambda_{t,x}$$
		for $\mu_t\,dt$-almost every $(t,x)$, and integrating with respect to $\mu_t\,dt$ gives (iii).
		
		\
		
		\noindent \emph{Proof of (ii).} Let $\varphi\in C_c^\infty((0,1)\times X)$. For $\mu_t\,dt$-almost every $(t,x)$ we have $\tilde v(t,x)\in d\iota_x(T_xX)$, and hence
		\begin{displaymath}
			\langle\nabla\varphi(t,x),v_t(x)\rangle_g=\bigl\langle d\iota_x\nabla\varphi(t,x),\tilde v(t,x)\bigr\rangle=\int\bigl\langle d\iota_x\nabla\varphi(t,x),w\bigr\rangle\,d\lambda_{t,x}.
		\end{displaymath}
		Since $\partial_t\varphi$ and $\|\nabla\varphi\|_g$ are bounded and $w\in L^1(\lambda)$, and since $\lambda_{t,x}$ is concentrated on $\Psi^{-1}(\{(t,x)\})$, it follows that
		\begin{align*}
			&\int_0^1\int_X\bigl(\partial_t\varphi+\langle\nabla\varphi,v_t\rangle_g\bigr)\,d\mu_t\,dt\\
			&\qquad=\int\Bigl(\partial_t\varphi(t,\omega(t))+\bigl\langle d\iota_{\omega(t)}\nabla\varphi(t,\omega(t)),w(t,\omega)\bigr\rangle\Bigr)\,d\lambda(t,\omega)\\
			&\qquad=\int_{C([0,1];X)}\int_0^1\Bigl(\partial_t\varphi(t,\omega(t))+\bigl\langle\nabla\varphi(t,\omega(t)),\dot\omega(t)\bigr\rangle_g\Bigr)\,dt\,d\boldsymbol\eta(\omega)\\
			&\qquad=\int_{C([0,1];X)}\int_0^1\frac{d}{dt}\varphi(t,\omega(t))\,dt\,d\boldsymbol\eta(\omega)=0,
		\end{align*}
		where we used the chain rule for the absolutely continuous curve $t\mapsto\varphi(t,\omega(t))$, which vanishes near $t=0$ and $t=1$. Together with the tangency of $v$ and the observation following Definition \ref{def:foliated-continuity-equation}, this proves (ii).
	\end{proof}
	
	\
	
	To relate the dynamical cost to the static foliated transport
	cost, we must also pass in the opposite direction: starting from
	a foliated transport plan, we need to construct an admissible
	evolution by joining each pair of transported points through a
	minimizing geodesic contained in their common leaf. Here, a constant-speed minimizing geodesic joining $x$ and $x'$ in their common leaf is an absolutely continuous curve $\alpha:[0,1]\to X$ with image in that leaf, $\alpha(0)=x$, $\alpha(1)=x'$, and $\|\dot\alpha(t)\|_g=d_{\mathcal L}(x,x')$ for almost every $t$. When $\mathcal F$ is smooth, the leaves are smooth Riemannian manifolds and these curves are the usual minimizing geodesics. In order to integrate these geodesics against a transport plan, they must be chosen measurably with respect to their endpoints. The following lemma shows that this is always possible on a complete manifold.
	
	\
	
	\begin{lemma} \label{lemma:borel-geodesic-selection}
		Let $X$ be a complete Riemannian manifold endowed with a Borel foliation $\mathcal F$ of class $C^1$. Then, there exists $G:E_{\mathcal F}\to C([0,1];X)$ Borel map, such that, for every $(x,x')\in E_{\mathcal F}$, the curve $G(x,x')$ is a constant-speed minimizing geodesic joining $x$ and $x'$ in their common leaf.
	\end{lemma}
	
	\begin{proof}
		Let $\mathcal C\subset C([0,1];X)$ be the set of continuous curves with image in a single leaf. We first show that $\mathcal C$ is closed. Let $\omega_n\in\mathcal C$ converge uniformly to $\omega$. Cover $[0,1]$ by finitely many closed intervals $J$, consecutive ones overlapping, such that each $\omega(J)$ is contained in the domain $U$ of a foliated chart $\phi$. Fix such an interval $J$. For $n$ large, $\omega_n(J)\subset U$, and then the transverse coordinate $\operatorname{pr}_2\circ\phi\circ\omega_n$ is constant on $J$. Hence so is its uniform limit $\operatorname{pr}_2\circ\phi\circ\omega$, and $\omega(J)$ lies in a single plaque. Since consecutive intervals overlap, $\omega\in\mathcal C$.
		
		\
		
		\noindent For $0\leq s<t\leq1$ and $\omega\in C([0,1];X)$, let $\operatorname{Length}_{s,t}(\omega)$ denote the length of the restriction $\omega|_{[s,t]}$, defined as in Section \ref{Section:Kantorovich}. Each $\operatorname{Length}_{s,t}$ is a supremum of continuous functions, hence lower semicontinuous on $C([0,1];X)$. Let $\Gamma$ be the set of triples $(x,x',\omega)\in E_{\mathcal F}\times C([0,1];X)$ such that $\omega(0)=x$, $\omega(1)=x'$, $\omega\in\mathcal C$, and
		\begin{displaymath}
			\operatorname{Length}_{s,t}(\omega)\leq(t-s)\,d_{\mathcal L}(x,x')
		\end{displaymath}
		for all rational $0\leq s<t\leq1$. Since $E_{\mathcal F}$ is Borel and $d_{\mathcal L}$ is Borel by Lemma \ref{lemma:dL-Borel-lsc}, the set $\Gamma$ is a Borel subset of $E_{\mathcal F}\times C([0,1];X)$. Fix $(x,x')\in E_{\mathcal F}$ and set $\ell:=d_{\mathcal L}(x,x')<+\infty$. We claim that the section $\Gamma_{(x,x')}$ is the set of constant-speed minimizing geodesics joining $x$ and $x'$ in their common leaf, and that it is nonempty and compact.
		
		\
		
		\noindent Let $\omega\in\Gamma_{(x,x')}$. Given real numbers $0\leq s<t\leq1$ and rationals $s'\leq s$ and $t'\geq t$, we have $d(\omega(s),\omega(t))\leq\operatorname{Length}_{s,t}(\omega)\leq\operatorname{Length}_{s',t'}(\omega)\leq(t'-s')\ell$. Letting $s'\to s$ and $t'\to t$ shows that $\omega$ is $\ell$-Lipschitz. Hence $\omega$ is absolutely continuous and $\|\dot\omega(t)\|_g\leq\ell$ for almost every $t$. Since $\omega$ is an absolutely continuous curve from $x$ to $x'$ with image in a leaf, the definition of $d_{\mathcal L}$ gives $\ell\leq\int_0^1\|\dot\omega(t)\|_g\,dt\leq\ell$. Therefore $\|\dot\omega(t)\|_g=\ell$ for almost every $t$, that is, $\omega$ is a constant-speed minimizing geodesic joining $x$ and $x'$ in their common leaf. Conversely, such a geodesic $\alpha$ belongs to $\mathcal C$ and satisfies
		$$\operatorname{Length}_{s,t}(\alpha)=\int_s^t\|\dot\alpha(r)\|_g\,dr=(t-s)\ell,$$
		so $\alpha\in\Gamma_{(x,x')}$. The section $\Gamma_{(x,x')}$ is closed in $C([0,1];X)$, since $\mathcal C$ is closed and each $\operatorname{Length}_{s,t}$ is lower semicontinuous. Its elements are $\ell$-Lipschitz curves starting at $x$, so they take values in the closed ball of radius $\ell$ centered at $x$, which is compact by the Hopf-Rinow theorem. By the Arzel\`a-Ascoli theorem, $\Gamma_{(x,x')}$ is compact.
		
		\
		
		\noindent To see that $\Gamma_{(x,x')}$ is nonempty, choose $\alpha_n:[0,1]\to X$ absolutely continuous curves  from $x$ to $x'$, each with image in a single leaf and of length at most $\ell+1/n$, parametrized proportionally to arc length. They are $(\ell+1)$-Lipschitz and take values in the compact closed ball of radius $\ell+1$ centered at $x$. By the Arzel\`a-Ascoli theorem, a subsequence converges uniformly to a curve $\alpha$. Then $\alpha\in\mathcal C$, $\alpha(0)=x$, $\alpha(1)=x'$, and, by lower semicontinuity,
		\begin{displaymath}
			\operatorname{Length}_{s,t}(\alpha)\leq\liminf_{n\to\infty}\operatorname{Length}_{s,t}(\alpha_n)=\liminf_{n\to\infty}(t-s)\operatorname{Length}(\alpha_n)\leq(t-s)\ell
		\end{displaymath}
		for all $0\leq s<t\leq1$. Thus $\alpha\in\Gamma_{(x,x')}$. Since $E_{\mathcal F}$ is a standard Borel space, $C([0,1];X)$ is Polish, and $\Gamma$ is a Borel set whose sections are nonempty and compact, $\Gamma$ admits a Borel uniformization by \cite[Theorem 28.8]{Kec95}. That is, there exists a Borel map $G:E_{\mathcal F}\to C([0,1];X)$ with $(x,x',G(x,x'))\in\Gamma$ for every $(x,x')\in E_{\mathcal F}$. By the claim, $G$ has the required properties.
	\end{proof}
	
	\
	
	This Borel selection allows a foliated transport plan $\gamma$ to be lifted to a probability measure $G_*\gamma$ on the space of leafwise trajectories. Combining this construction with the superposition principle yields the equivalence between the static and dynamical formulations, as stated in the following foliated Benamou-Brenier formula, which is Theorem \ref{maintheorem:B} from the introduction.
	
	\
	
	\begin{theorem}	\label{theorem:foliated-benamou-brenier}
		Let $X$ be a complete Riemannian manifold endowed with a
		Borel foliation $\mathcal F$ of class $C^1$. Then, for every probability measure $\mu_0$ and $\mu_1\in\mathscr P_2(X)$, the dynamical and static foliated transport costs coincide, that is, $\mathcal W_{2,\mathcal F}^{2}(\mu_0,\mu_1) = W_{2,\mathcal F}^{2}(\mu_0,\mu_1)$.
	\end{theorem}
	
	\begin{proof}
		We first prove that $W_{2,\mathcal F}^{2}(\mu_0,\mu_1) \leq \mathcal W_{2,\mathcal F}^{2}(\mu_0,\mu_1)$.
		
		\
		
		\noindent Let $(\mu_t,v_t) \in	\operatorname{CE}_{\mathcal F}(\mu_0,\mu_1)$. By Lemma \ref{lemma:leafwise-superposition}, there exists $\boldsymbol{\eta} \in \mathscr P\bigl(C([0,1];X)\bigr)$ such that $(e_t)_*\boldsymbol{\eta}=\mu_t$ for every $t\in[0,1]$, and $\boldsymbol{\eta}$ is concentrated on absolutely continuous curves satisfying $\dot\omega(t)=v_t(\omega(t))$ for almost every $t$, each of which is contained in a single leaf. Define $\gamma:=(e_0,e_1)_*\boldsymbol{\eta}$. The marginals of $\gamma$ are $\mu_0$ and $\mu_1$, so that $\gamma\in\Pi(\mu_0,\mu_1)$. By Lemma \ref{lemma:leafwise-superposition}(iii), $\pi(\omega(0))=\pi(\omega(1))$ for $\boldsymbol{\eta}$-almost every $\omega$, and hence $\gamma(E_{\mathcal F})=1$. Thus, $\gamma\in\Pi_{\mathcal F}(\mu_0,\mu_1)$. Moreover, since $\boldsymbol{\eta}$-almost every $\omega$ is an absolutely continuous curve contained in a single leaf, the definition of $d_{\mathcal L}$ gives
		\begin{displaymath}
			d_{\mathcal L}(\omega(0),\omega(1))
			\leq
			\int_0^1\|\dot\omega(t)\|_g\,dt.
		\end{displaymath}
		By the Cauchy-Schwarz inequality,
		\begin{displaymath}
			d_{\mathcal L}^2(\omega(0),\omega(1))
			\leq
			\int_0^1\|\dot\omega(t)\|_g^2\,dt.
		\end{displaymath}
		Integrating with respect to $\boldsymbol{\eta}$ gives
		\begin{align*}
			\int_{X\times X}d_{\mathcal L}^2(x,x')\,d\gamma(x,x')
			&=
			\int_{C([0,1];X)}
			d_{\mathcal L}^2(\omega(0),\omega(1))
			\,d\boldsymbol{\eta}(\omega)\\
			&\leq
			\int_{C([0,1];X)}
			\int_0^1\|\dot\omega(t)\|_g^2\,dt\,
			d\boldsymbol{\eta}(\omega)\\
			&=
			\int_0^1\int_X
			\|v_t(x)\|_g^2\,d\mu_t(x)\,dt\\
			&=
			\mathcal A_{\mathcal F}(\mu,v).
		\end{align*}
		Therefore, $W_{2,\mathcal F}^{2}(\mu_0,\mu_1) \leq \mathcal A_{\mathcal F}(\mu,v)$. Taking the infimum over all admissible pairs yields
		\begin{equation}
			W_{2,\mathcal F}^{2}(\mu_0,\mu_1)
			\leq
			\mathcal W_{2,\mathcal F}^{2}(\mu_0,\mu_1).
			\label{eq:static-below-dynamic}
		\end{equation}
		
		\
		
		\noindent We now prove the reverse inequality. If $W_{2,\mathcal F}^{2}(\mu_0,\mu_1)=+\infty$, then inequality \eqref{eq:static-below-dynamic} implies $\mathcal W_{2,\mathcal F}^{2}(\mu_0,\mu_1)=+\infty$, and there is nothing to prove. Suppose that $W_{2,\mathcal F}^{2}(\mu_0,\mu_1)<+\infty$. Fix $\varepsilon>0$ and choose
		$\gamma_\varepsilon\in\Pi_{\mathcal F}(\mu_0,\mu_1)$ such that
		\begin{displaymath}
			\int_{X\times X}
			d_{\mathcal L}^{2}(x,x')\,d\gamma_\varepsilon(x,x')
			\leq
			W_{2,\mathcal F}^{2}(\mu_0,\mu_1)+\varepsilon.
		\end{displaymath}
		
		\
		
		\noindent Using the Borel selection $G$ of leafwise minimizing
		geodesics given by Lemma \ref{lemma:borel-geodesic-selection}, define $\boldsymbol{\eta}_\varepsilon := G_*\gamma_\varepsilon \in	\mathscr P\bigl(C([0,1];X)\bigr)$ and $\mu_t^\varepsilon := (e_t)_*\boldsymbol{\eta}_\varepsilon$. Since $G(x,x')$ is contained in the common leaf of $x$ and $x'$, the measure $\boldsymbol{\eta}_\varepsilon$ is concentrated on leafwise curves. Moreover, $\mu_0^\varepsilon=\mu_0$, $\mu_1^\varepsilon=\mu_1$.
		
		\
		
		\noindent Every curve $G(x,x')$ is a constant-speed minimizing geodesic inside the common leaf of $x$ and $x'$, so that $\|\frac{d}{dt}G(x,x')(t)\|_g=d_{\mathcal L}(x,x')$ for almost every $t$. Hence,
		\begin{displaymath}
			\int_0^1
			\left\|
			\frac{d}{dt}G(x,x')(t)
			\right\|_g^2\,dt
			=
			d_{\mathcal L}^{2}(x,x'),
		\end{displaymath}
		and integrating with respect to $\gamma_\varepsilon$ gives
		\begin{displaymath}
			\int_{C([0,1];X)}\int_0^1\|\dot\omega(t)\|_g^2\,dt\,d\boldsymbol{\eta}_\varepsilon(\omega)
			=
			\int_{X\times X}d_{\mathcal L}^{2}(x,x')\,d\gamma_\varepsilon(x,x')
			<+\infty.
		\end{displaymath}
		Since $\mu_0\in\mathscr P_2(X)$, Lemma \ref{lemma:eulerian-representation} applied to $\boldsymbol{\eta}_\varepsilon$ provides a Borel vector field $v^\varepsilon$ such that $(\mu_t^\varepsilon,v_t^\varepsilon)$ satisfies the foliated continuity equation, $\mu_t^\varepsilon\in\mathscr P_2(X)$ for every $t$, and
		\begin{displaymath}
			\mathcal A_{\mathcal F}(\mu^\varepsilon,v^\varepsilon)
			\leq
			\int_{C([0,1];X)}\int_0^1\|\dot\omega(t)\|_g^2\,dt\,d\boldsymbol{\eta}_\varepsilon(\omega).
		\end{displaymath}
		In particular, $(\mu_t^\varepsilon,v_t^\varepsilon)\in\operatorname{CE}_{\mathcal F}(\mu_0,\mu_1)$. It follows that
		\begin{align*}
			\mathcal W_{2,\mathcal F}^{2}(\mu_0,\mu_1)
			&\leq
			\mathcal A_{\mathcal F}
			(\mu^\varepsilon,v^\varepsilon)\\
			&\leq
			\int_{X\times X}
			d_{\mathcal L}^{2}(x,x')\,d\gamma_\varepsilon(x,x')\\
			&\leq
			W_{2,\mathcal F}^{2}(\mu_0,\mu_1)+\varepsilon.
		\end{align*}
		Since $\varepsilon>0$ is arbitrary, we conclude that
		\begin{equation}
			\mathcal W_{2,\mathcal F}^{2}(\mu_0,\mu_1)
			\leq
			W_{2,\mathcal F}^{2}(\mu_0,\mu_1).
			\label{eq:dynamic-below-static}
		\end{equation}
		
		\
		
		\noindent Combining \eqref{eq:static-below-dynamic} and
		\eqref{eq:dynamic-below-static}, we obtain $\mathcal W_{2,\mathcal F}^{2}(\mu_0,\mu_1) = W_{2,\mathcal F}^{2}(\mu_0,\mu_1)$.
	\end{proof}
	
	\

		As anticipated in the introduction, the foliation restricts which points may be coupled in the static formulation, and which directions of motion are admissible in the dynamical one. The following elementary example illustrates this equivalence explicitly.
		
		\
		
		\begin{example}
			Let $X=\mathbb T^2=\mathbb S^1\times\mathbb S^1$ be the flat two-dimensional torus, where $\mathbb S^1=\mathbb R/\mathbb Z$, endowed with the foliation whose leaves are the horizontal circles $\mathcal L_y=\mathbb S^1\times\{y\}$, $y\in\mathbb S^1$. The corresponding quotient map is $\pi:\mathbb T^2\to\mathbb S^1$, $\pi(x,y)=y$. Let $\sigma\in\mathscr P(\mathbb S^1)$ and fix
			$a\in[-1/2,1/2]$. Consider the probability measures $\mu_0:=\delta_0\otimes\sigma$, $\mu_1:=\delta_a\otimes\sigma$. Thus, on every horizontal leaf, the initial mass is concentrated
			at $(0,y)$ and the final mass is concentrated at $(a,y)$. Clearly, $\pi_*\mu_0=\pi_*\mu_1=\sigma$, so the transverse compatibility condition is satisfied.
			
			\
			
			\noindent The disintegrations of $\mu_0$ and $\mu_1$ with respect to $\pi$ are given by $\mu_0^y=\delta_{(0,y)}$, $\mu_1^y=\delta_{(a,y)}$ for $\sigma$-almost every $y\in\mathbb S^1$. Since the intrinsic
			distance on each horizontal leaf is the usual distance on
			$\mathbb S^1$, we have 
			$$W_{2,\mathcal L_y}^2(\mu_0^y,\mu_1^y) =	d_{\mathbb S^1}^2(0,a) =|a|^2.$$
			Therefore, the leafwise decomposition formula gives
			\begin{align*}
				W_{2,\mathcal F}^2(\mu_0,\mu_1)
				&=
				\int_{\mathbb S^1}
				W_{2,\mathcal L_y}^2(\mu_0^y,\mu_1^y)\,d\sigma(y)\\
				&=
				\int_{\mathbb S^1}|a|^2\,d\sigma(y)\\
				&=
				|a|^2.
			\end{align*}
			
			\
			
			\noindent We now recover the same value from the dynamical formulation. Define $\mu_t:=\delta_{ta}\otimes\sigma$, $t\in[0,1]$, where $ta$ is understood modulo $\mathbb Z$, and consider $v_t(x,y):=a\,\partial_x$. The vector field $v_t$ is tangent to the horizontal leaves, since $T_{(x,y)}\mathcal L_y =	\operatorname{span}\{\partial_x\}$.	Moreover, the pair $(\mu_t,v_t)$ satisfies the foliated continuity equation and connects $\mu_0$ to $\mu_1$. Its kinetic action is
			\begin{align*}
				\mathcal A_{\mathcal F}(\mu,v)
				&=
				\int_0^1\int_{\mathbb T^2}
				\|v_t(x,y)\|_g^2\,d\mu_t(x,y)\,dt\\
				&=
				\int_0^1\int_{\mathbb S^1}|a|^2\,d\sigma(y)\,dt\\
				&=
				|a|^2.
			\end{align*}
			Consequently, $\mathcal W_{2,\mathcal F}^2(\mu_0,\mu_1) \leq |a|^2$. By Theorem \ref{theorem:foliated-benamou-brenier}, the reverse inequality also holds, and
			\begin{displaymath}
				\mathcal W_{2,\mathcal F}^2(\mu_0,\mu_1)=W_{2,\mathcal F}^2(\mu_0,\mu_1)=|a|^2.
			\end{displaymath}
			
			\
			
			\noindent This example makes the transverse conservation law explicit: the measure $\sigma$ describing the distribution of mass among the horizontal leaves remains fixed, while the conditional mass on every leaf moves horizontally from $0$ to $a$ with constant speed.
		\end{example}
		
		\
		
		More generally, the foliated Benamou-Brenier formula allows the transport on different leaves to follow distinct dynamics, provided that these dynamics depend measurably on the transverse parameter. The foliation therefore separates the evolution into independent leafwise motions while the common transverse marginal keeps these motions assembled into a global probability measure on $X$. This provides the dynamical counterpart of the leafwise decomposition established in the static formulation.

	\
	
	\section*{Acknowledgements}
	\noindent This work was financed in part by the Coordena\c{c}\~ao de Aperfei\c{c}oamento de Pessoal de N\'ivel Superior - Brasil (CAPES) - Finance Code 001. This work was financed in part by Universidade Estadual de Campinas (UNICAMP) through funds provided under Resolution GR-03/2026. 
	
	\
	
	\noindent D. S. L. has been partially supported by S\~{a}o Paulo Research Foundation (FAPESP) grant \#2022/03379-0. C. S. R. has been partially supported by S\~{a}o Paulo Research Foundation (FAPESP) grant \#2020/04426-6. The opinions, hypotheses and conclusions or recommendations expressed in this work are the responsibility of the authors and do not necessarily reflect the views of FAPESP. C. S. R. would like to acknowledge support from the Max Planck Society, Germany, through the award of a Max Planck Partner Group for Geometry and Probability in Dynamical Systems. 
	
	\

\end{document}